\documentclass[reqno]{amsart}
\usepackage{amsmath, amsfonts, amssymb, latexsym, amsthm, amscd, cancel,
  enumerate, rotating, comment, appendix, hyperref,soul,
  times,graphicx,tikz,tikz-cd,pgf}
\usepackage[capitalise, noabbrev]{cleveref}
\usepackage[all]{xy}
\usepackage[latin1]{inputenc}
\usepackage[english]{babel}
\usepackage[mathscr]{eucal}

\usetikzlibrary{calc,arrows,snakes,shapes,matrix, decorations.pathmorphing}
\tikzset{>=latex}

\newtheorem{theorem}{Theorem}[section]
\newtheorem{lemma}[theorem]{Lemma}
\newtheorem{corollary}[theorem]{Corollary} 
\newtheorem{proposition}[theorem]{Proposition}

\theoremstyle{definition}
\newtheorem{definition}[theorem]{{Definition}}
\newtheorem{example}[theorem]{Example}
\newtheorem{remark}[theorem]{Remark}
\newtheorem{discussion}[theorem]{Discussion}
\newtheorem{notation}[theorem]{Notation}
\newtheorem{terminology}[theorem]{Terminology}
\newtheorem{construction}[theorem]{Construction}

\theoremstyle{remark}
\newtheorem*{claim}{Claim}

\numberwithin{equation}{theorem}

\newif\ifdviwin

\newcommand{\kk}{\textup{k}}
\newcommand{\m}{{\mathfrak m}}

\newcommand{\dstab}{\operatorname{dstab}}
\newcommand{\reg}{\operatorname{reg}}

\newcommand{\im}{\operatorname{im}}
\newcommand{\ba}{\mathbf{a}}
\newcommand{\bb}{\mathbf{b}}
\newcommand{\bc}{\mathbf{c}}

\newcommand{\bm}{{\mathbf{m}}}
\newcommand{\bma}{{\bm^{\ba}}}
\newcommand{\bmb}{{\bm^{\bb}}}
\newcommand{\bmc}{{\bm^{\bc}}}

\newcommand{\bpi}{{\pmb{\pi}}}
\renewcommand{\bpi}{\Pi}

\newcommand{\pd}{\operatorname{pd}}
\newcommand{\gr}{\operatorname{gr}}

\newcommand{\cM}{\mathcal M}
\newcommand{\cN}{\mathcal N}
\newcommand{\supp}{\operatorname{Supp}}
\newcommand{\osigma}{\overline{\sigma}}

\newcommand{\lcm}{\operatorname{lcm}}

\newcommand{\NN}{\mathbb{N}}
\newcommand{\F}{\mathcal{F}}
\newcommand{\dom}{\mathfrak{d}}
\newcommand{\LCM}{\operatorname{LCM}}
\newcommand{\taylor}{\operatorname{Taylor}}
\newcommand{\ssm}{\smallsetminus}
\newcommand{\qand}{\quad \mbox{ and } \quad}
\newcommand{\qor}{\quad \mbox{ or } \quad}
\newcommand{\qif}{\quad \mbox{ if } \quad}

\begin{document}
\bibliographystyle{amsplain}

\author[S.M.~Cooper]{Susan M. Cooper}
\address{Department of Mathematics\\
University of Manitoba\\
520 Machray Hall\\
186 Dysart Road\\
Winnipeg, MB\\
Canada R3T 2N2}
\email{susan.cooper@umanitoba.ca}

\author[S.~El Khoury]{Sabine El Khoury}
\address{Department of Mathematics,
American University of Beirut,
Bliss Hall 315, P.O. Box 11-0236,  Beirut 1107-2020,
Lebanon}
\email{se24@aub.edu.lb}

\author[S.~Faridi]{Sara Faridi}
\address{Department of Mathematics \& Statistics\\
Dalhousie University\\
6316 Coburg Rd.\\
PO BOX 15000\\
Halifax, NS\\
Canada B3H 4R2}
\email{faridi@dal.ca}

\author[S~Mayes-Tang]{Sarah Mayes-Tang}
\address{Department of Mathematics\\
University of Toronto\\
40 St. George Street, Room 6290\\
Toronto, ON \\
Canada M5S 2E4}
\email{smt@math.toronto.edu}

\author[S.~Morey]{Susan Morey}
\address{Department of Mathematics\\
Texas State University\\
601 University Dr.\\
San Marcos, TX 78666\\USA}
\email{morey@txstate.edu}

\author[L.~M.~\c{S}ega]{Liana M.~\c{S}ega}
\address{Department of Mathematics and Statistics\\
   University of Missouri - Kansas~City\\  Kansas~City\\ MO 64110\\ USA}
     \email{segal@umkc.edu}

\author[S.~Spiroff]{Sandra Spiroff }
\address{Department of Mathematics,
University of Mississippi,
Hume Hall 335, P.O. Box 1848, University, MS 38677
USA}
\email{spiroff@olemiss.edu}

\subjclass[2010]{Primary 13A15, 13D02, 05E40; Secondary 13C15}
\keywords{monomial ideal,  projective dimension one, powers of ideal, minimal resolution, Discrete Morse Theory}
\title[Morse Resolutions of Powers]
{Morse resolutions of powers of square-free monomial ideals 
of projective dimension one}

\begin{abstract} 
  Let $I$ be a square-free monomial ideal of
   projective dimension one. Starting with the Taylor complex on the
   generators of $I^r$, we use Discrete Morse theory to describe a CW
   complex that supports a minimal free resolution of $I^r$. To do so,
   we concretely describe the acyclic matching on the faces of the
   Taylor complex.
\end{abstract}

\maketitle

\section{Introduction}

The powers of an ideal $I$ in a ring $R$ have a
  fascinating but poorly understood structure.  Describing these
  powers, and algebraic invariants associated to them, is a vibrant
  area of mathematics, with problems and tools arising from algebra,
  geometry, and combinatorics. Even when the ideal
    itself is well understood, the powers typically have unexpected
    behavior, making their classification difficult.
    
In general there is more known about the asymptotic behavior of the
collection $\{I^r \}$ as $r \rightarrow \infty$ than about the
individual ideals $I^r$, see M.~Brodmann~\cite{B79},
V. Kodiyalam~\cite{K93}, S.~D.~Cutokosky, J.~Herzog, and
N.~V.~Trung~\cite{CDHT} and Kodiyalam~\cite{K00}.  On the other hand,
relatively little is known about the structure of the resolutions of
the individual ideals $I^r$. Although many invariants of $I$ are known
to stabilize for large powers of $I$, these invariants can exhibit
poor behavior for powers below the stable point.

  In the case where the ideal $I$ is a monomial ideal (and hence so
  are all its powers $I^r$), the problem of finding free resolutions
  can be translated into finding combinatorially-described topological
  objects whose chain maps can be adapted to provide a free resolution
  of the ideal. D.~Taylor's thesis~\cite{T} initiated this approach by
  encoding the faces of a simplex with the least
  common multiples of the monomial generators of the ideal. Over the
  last few decades, Taylor's construction has been generalized
  (\cite{BS,P,BW}) to other topological objects using the same idea:
  label the vertices of a topological object $\Gamma$ with monomials
  $m_1,\ldots,m_q$ and look for conditions under which the cellular
  chain complex of $\Gamma$ can be homogenized to describe a free
  resolution of $I=(m_1,\ldots,m_q)$. In this case we say $\Gamma$
  {\bf supports a free resolution} of $I$, and $I$ has a cellular
  resolution, and we use combinatorial and homological information
  about $\Gamma$ to gain information about invariants of $I$.

This point of view allows one to use homotopy theory to describe free
resolutions of ideals as the cellular resolutions of homotopy
equivalent topological objects. Discrete Morse Theory - an adaptation
of Morse theory of manifolds to the discrete setting - is one such
homotopy theory.  It was proven by R.~Forman~\cite{For} that a discrete
Morse function $f$ on the set of cells of a CW complex $X$ produces a
CW complex $X_f$ which is homotopy equivalent to
$X$. M.~Chari~\cite{Char} reinterpreted Forman's Morse function as
acyclic matchings. E.~Batzies and V.~Welker~\cite{BW,Ba} showed that if the
matchings are homogeneous, then this construction leads to a
multigraded resolution. As demonstrated in~\cite{BW}, Morse theory has
great potential in the study of resolutions, providing a clever tool
to shrink large free resolutions into smaller ones, see also \cite{BM}, \cite{EN}, and \cite[Chapter 6]{JJ09}. It is important to
stress that not all monomial ideals have \emph{minimal} cellular
resolutions; see M.~Velasco~\cite{Ve}.

Inspired by recent work of J.~\`{A}lvarez Montaner, O.~
  Fern\'{a}ndez-Ramos, P.~Gimenez ~\cite{spanish} where
  acyclic matchings are used to prune extra faces from the Taylor complex in
  order to achieve a smaller resolution for a given monomial ideal, 
our approach in this paper is to define homogeneous acyclic matchings
on the poset lattice of the Taylor complex of $I^r$ for a
square-free monomial ideal  $I$ of projective dimension one.

The  structure of ideals of projective
  dimension one has an illustrious history. In 1890
  D.~Hilbert~\cite{hilbert}, in a result extended by
  L.~Burch~\cite{burch}, described the structure of these ideals in
  terms of the minors of a presentation matrix, resulting in the
  celebrated Hilbert-Burch theorem. (See \cite[Theorem
    20.15]{eisenbud} for the statement and historical context.)
  
Our motivation for this paper is  the classification of
resolutions of monomial ideals of projective dimension one in
B. Hersey's Master's thesis (\cite{FH,H}) as cellular resolutions
supported on graphs.  In this paper, we exploit the structure of the
graph supporting a minimal free resolution of $I$ to find a CW complex
that supports a free resolution of $I^r$. In addition, we give
concrete descriptions of the cells and gradient paths of the resulting
Morse complex to show that the resolution is minimal.

 Our main results can be summarized in the following statement.

\begin{theorem}[\cref{mainresult,t:minimality}]\label{main}
	If $I$ is an ideal of projective dimension one in a polynomial
        ring over a field, minimally generated by $q$ square-free
        monomials, and $r$ is a positive integer, then there exists a
        homogeneous acyclic matching on the Taylor complex of $I^r$
        (that is explicitly defined using the structure of a graph
        supporting the resolution of $I$) such that the resulting
        CW complex supports a minimal free resolution of $I^r$.
\end{theorem}

  As an immediate corollary, we recover an explicit formula for the
  projective dimensions, or equivalently the depths, of the powers of
  $I$ (\cref{c:pd-I^r,c:dstab}), and we are able to
    describe the regularity of large powers of $I$ in \cref{c:reg}.

While the focus of this paper is on Morse resolutions, it is worth
noting that in the case of monomial ideals of projective dimension
one, these minimal resolutions can also be obtained from a Koszul
complex, see~\cite{koszul}.  Our work here, however, aligns with a
different type of question (see also Engstr\"{o}m and
Noren~\cite{EN}): \emph{Starting from a (minimal) free resolution of
  $I$ supported on a cell (simplicial) complex $X$, how close to a
  minimal free resolution of $I^r$ can we get by using homogeneous acyclic
  matchings to reach a smaller CW complex $X^r$ supporting a free
  resolution of $I^r$?}

The paper is organized as follows. \cref{s:setup} provides the basic
definitions of simplicial and CW complexes, cellular resolutions, and
the particular case of ideals of projective dimension one.  In
\cref{s:matchings} we develop the Morse matching on the faces of the
Taylor complex, which will later be used in \cref{mainsection} to
construct the Morse complex supporting a free resolution of $I^r$.  In
\cref{s:ppd1} we develop the necessary monomial labelings for the
Morse complex, and in \cref{s:morse} we describe the cells of the
Morse complex supporting the resolution, all of which will allow us to
prove, in \cref{minimalsection}, that the cellular resolution of the
Morse complex homogenizes to a minimal free resolution of $I^r$.

  \section{Setup}\label{s:setup}

\subsection{Simplicial complexes and quasi-trees}
 We begin by recalling some standard definitions and notations. A simplicial complex $\Delta$ on a vertex set $V$ is a set of subsets
of $V$ such that $F \in \Delta$ and $G \subseteq F$ implies that $G
\in \Delta$. An element $F$ of $\Delta$ is called a {\bf face}, and a
maximal face of $\Delta$ (under inclusion) is called a {\bf facet}. A
simplicial complex $\Delta$ can be uniquely determined by its facets
$F_1,\ldots,F_q$, so we write $\Delta=\langle  F_1,\ldots,F_q\rangle.$
  A {\bf simplex} is a simplicial complex with only one facet.
 The {\bf dimension} of a face $F \in \Delta$ is $\dim (F)= |F|-1$, and the
  {\bf dimension of $\Delta$} is $\dim(\Delta)=\max\{\dim(F) \mid F
  \in \Delta\}$. The $0$-dimensional faces are called {\bf vertices}
  and the face $\varnothing$ has dimension $-1$.

  A facet $F$ of $\Delta$ is a {\bf leaf}~\cite{F02} if $F$ is the
  only facet of $\Delta$ or if there is another facet $G$ of $\Delta$,
  called a {\bf joint} such that $(F \cap H) \subseteq G$ for every
  facet $H \neq F$.
  A {\bf free vertex} of $\Delta$ is a vertex belonging to exactly one
  facet of $\Delta$.  If $F$ is a leaf of a simplicial complex, then
  $F$ necessarily has a free vertex~\cite{F02}.
  A {\bf quasi-forest}~\cite{Z} is a simplicial complex $\Delta$ whose
  faces can be\ ordered as $F_1,\ldots,F_q$ such that for each $i \in
  \{1,\ldots,q\}$, the facet $F_i$ is a leaf of $\langle
  F_1,\ldots,F_i \rangle$.  A connected quasi-forest is called a {\bf
    quasi-tree}.
    
  If $\Delta=\langle F_1,\ldots,F_q \rangle$ is a simplicial complex
  on the vertex set $V$, then the {\bf complement of $\Delta$} is the
  simplicial complex $$\Delta^c=\langle F_1^c,\ldots,F_q^c \rangle$$
  where $F_i^c= V \ssm F_i$.

    These definitions give a variety of ways to construct square-free monomial ideals by means of simplicial complexes and vice versa. While traditionally a simplicial complex was associated to its Stanley-Reisner ring \cite[Chapter II]{S}, and more recently many authors have studied edge ideals, defined in \cite{V} (or more generally facet ideals defined in \cite{F02}), the focus of this paper will be to exploit the connections between the square-free monomial generators of an ideal $I$ and the structure of $\Delta^c$.

\subsection{Resolutions supported on simplicial complexes}
\label{ss:supported}

Assume that $S=\kk[x_1, \ldots, x_n]$ is a polynomial ring over a
field $\kk$ and $I=(m_1,\ldots,m_q)$ is an ideal generated by
monomials.  Let $\LCM(I)$ denote the $\lcm$-lattice of $I$, that is,
the poset consisting of the $\lcm$'s of the generating set
$m_1,\ldots,m_q$ ordered by divisibility.

A {\bf graded free resolution of $I$} is an exact
sequence of free $S$-modules of the form:
$$ \mathbb{F} \colon 0 \to G_d \to \cdots \to G_i
\stackrel{\partial_i}{\longrightarrow} G_{i-1} \to \cdots \to G_1
\stackrel{\partial_1}{\longrightarrow} G_0$$ where $I \cong
G_0/\im(\partial_1)$, and each map $\partial_i$ is graded, in the
sense that it preserves the degrees of homogeneous elements.

If $\partial_i(G_i) \subseteq (x_1,\ldots,x_n) G_{i-1}$ for every
$i>0$, then the free resolution $\mathbb{F}$ is {\bf minimal}. The
grading on each of the free modules $G_i$ can be further refined by
writing $G_i$ as a direct sum of $1$-dimensional free $S$-modules of
the form $S(m)$ indexed by monomials $m$. In particular, when
$\mathbb{F}$ is a minimal resolution, $$G_i \cong \bigoplus_{m\in
  \LCM(I)} S(m) ^{\beta_{i,m}}$$ where the $\beta_{i,m}$ are
invariants of $I$ called the {\bf multigraded Betti numbers} of $I$.
The length of a minimal free resolution of $I$ ($d$ in the case of
$\mathbb{F}$ above) is another invariant of $I$ called the {\bf
  projective dimension} and denoted by $\pd_S(I)$.

One concrete way to calculate a multigraded free resolution is to use
chain complexes of topological objects. This approach was initiated
by Taylor~\cite{T}, and further developed by Bayer and Sturmfels~\cite{BS}
and many other researchers.

Given a monomial ideal $I=(m_1,\ldots,m_q)$, the {\bf Taylor complex}
$\taylor(I)$ of $I$ is a simplex on $q$ vertices, each of which is labeled by
a monomial generator of $I$, and where each face is labeled by the lcm of
the monomial labels of its vertices.

Below we provide a simple example of a Taylor complex, chosen
as the smallest interesting one. It will be used as a running example
throughout the paper
for the purpose of demonstrating the subtleties in the later
constructions via the diagrams and complexes associated to this ideal
and its powers.  For the ideal $I$ below, even $\taylor(I^2)$, for
instance, is already quite complicated. (See Example
\ref{shadedtriangle}.)

\begin{example}\label{e:running} If $I=(xy,yz,zu)$, then the 
  Taylor complex of $I$ is below. 

  $$
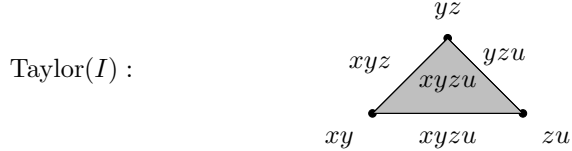
\begin{tikzpicture}
\coordinate (Z) at (-4, 1);
\coordinate (A) at (0, 0);
\coordinate (B) at (1, 1);
\coordinate (C) at (2, 0);
\coordinate (D) at (.5, .65);
\coordinate (E) at (1, 0);
\coordinate (F) at (1.75, .4);
\coordinate (G) at (1, .75);
\draw[black, fill=black] (A) circle(0.05);
\draw[black, fill=black] (B) circle(0.05);
\draw[black, fill=black] (C) circle(0.05);
\draw[fill=lightgray] (B) -- (C) -- (A) -- cycle;
\draw[-] (A) -- (B);
\draw[-] (B) -- (C);
\draw[-] (A) -- (C);
\node[label = below left:$xy$] at (A) {};
\node[label = above :$yz$] at (B) {};
\node[label = below right:$zu$] at (C) {};
\node[label = left :$xyz$] at (D) {};
\node[label = below :$xyzu$] at (E) {};
\node[label = above :$yzu$] at (F) {};
\node[label = below :$xyzu$] at (G) {};
\node[label = below :$\taylor(I):$] at (Z) {};
\end{tikzpicture}
$$
\end{example}

When $I$ is generated by $q$ monomials, one can observe that the
simplicial chain complex of $\taylor(I)$ (which is transformed into a
free resolution of $I$) has length equal to the dimension of
$\taylor(I)$, which is $(q-1)$. Since every free resolution contains a
minimal free resolution, we can see that $\pd_S(I) \leq q-1$.

Given $I$ generated by $q$ monomials, one could take any subcomplex
$\Gamma$ of the Taylor complex which includes all $q$ vertices and ask
whether its simplicial (or CW) chain complex could give a free
resolution of $I$ (we refer the interested reader to the book~\cite{P}
for a detailed exposition and further references in this area). If the
chain complex of $\Gamma$ produces a free resolution of $I$, we say
$\Gamma$ {\bf supports a resolution} of $I$, and in particular we will
have $\pd_S(I) \leq \dim(\Gamma) \leq q-1$. For $\Gamma$ to support a
minimal free resolution, we would need at the very least that
$\dim(\Gamma)=\pd_S(I)$.

\subsection{Ideals of projective dimension $1$}

If $\Gamma$ supports a minimal free resolution of $I$ and
$\pd(I)=1$,  then $\dim (\Gamma) = 1$ and so $\Gamma$ is an acyclic graph and hence a tree. Faridi and Hersey
proved that these two conditions are in fact equivalent.

Before stating the theorem, we recall a standard definition.
Let $S=\kk[x_1, \ldots, x_n]$ and $\Delta$ be a simplicial complex
whose vertices are labeled with the variables $x_1,\ldots,x_n$. Then
the square-free monomial ideal $$I=(x_{i_1}\cdots x_{i_r} \mid
\{x_{i_1},\ldots,x_{i_r}\} \mbox{ is a facet of } \Delta)$$ is called
the {\bf facet ideal of $\Delta$}, denoted by $\F(\Delta)$, and
$\Delta$ is called the {\bf facet complex of $I$}, denoted by $\F(I)$.
{Similarly, $\F(I)^c$ denotes the complex with facets $F_i^c$ where $F_i$ are the facets of $\F(I)$.

\begin{theorem}[{\cite[Theorem 27]{FH}}]\label{t:FH} 
Let $I$ be a square-free monomial ideal 
in a polynomial ring $S$. Then the following statements are equivalent.
  \begin{enumerate}  
    \item $\pd_S(I) \leq 1$;
    \item $\F(I)^c$ is a quasi-forest;
    \item The Alexander dual of the Stanley-Reisner complex of $I$ is a quasi-forest;
    \item $S/I$ has a minimal free resolution supported on a graph,
      which is a tree.
  \end{enumerate}
\end{theorem}

 This theorem will be used to associate a graph in part (4) to a monomial ideal of projective dimension one in part (1) and its associated complex in part (2). Part (3) is included to show connections to the broader field, including the theory of Stanley-Reisner. The graph in part~(4) of \cref{t:FH} is constructed
  based on an ordering of facets of the quasi-forest $\F(I)^c$ in
  part~(2), or equivalently an ordering of the minimal monomial
  generating set of $I$.  In fact, as seen in \cref{construct} the minimal generators of $I$ are obtained from  the product of the variables that are in the complement of the facets. Since the ordering of facets in a quasi-tree requires that every facet introduces a new variable,  it follows that the number of generators of $I$ is less or equal than $n$, see also \cref{x}. Furthermore, \cref{construct} is a procedure whose output
  is the ordered monomial generating set $m_1,\ldots,m_q$ for $I$,
  which we will use throughout the paper to build a cellular
  resolution for $I^r$. \cref{construct} also gives the definition of the {\bf joint function} $\tau$ which will be used heavily in the rest of the paper.

\begin{construction}[{\bf An order on the generators of $I$}]\label{construct}
  Let $I$ be a square-free monomial ideal  with $\pd_S(I)\le 1$.
\begin{enumerate}
\item Order the facets of $\Delta=\F(I)^c$ as $F_1, F_2,
  \dots, F_q$ such that $F_i$ is a leaf of $\Delta_i=\langle F_1,
  \dots, F_i\rangle$.
\item Start with the one vertex tree $T_1=(V_1, E_1)$ where 
$V_1=\{v_1\}$ and $E_1=\varnothing$. 
\item If $i=1$, set $\tau(1)=1$. For $i=2,\dots, q$ do the following: 
\begin{enumerate}[$\bullet$]
\item Pick $u < i$ such that $F_u$ is a joint of $F_i$ in
  $\Delta_i$. {\bf Set $\tau(i)=u$};
\item Set $V_i=V_{i-1}\cup \{v_i\}$;
\item Set $E_i=E_{i-1}\cup \{ (v_i, v_u)\}$. 
\end{enumerate}
\item The result is a tree $T=(V_q, E_q)$ with $q$ vertices. We label the vertex $v_i$ of
  $T$ with the monomial
$$
m_i=\prod_{x_t\in F_i^c}x_t.
$$ 
\end{enumerate} The monomials $m_1,\ldots,m_q$ form a
  minimal generating set of $I$, ordered as in Step~(1) above.
\end{construction}

\begin{example}\label{e:running1} If  $I=(xy,yz,zu)$ is the 
  ideal in \cref{e:running}, it is easy to check using a software such as Macaulay 2 \cite{M2} that $\pd_S(I)=1$. Set
  $\Delta = \F(I)^c$, which is a simplicial complex with facets
  $\{z,u\}$, $\{x,u\}$, $\{x,y\}$. In fact, $\Delta$ is a
  quasi-tree. Consider the facet order on $\Delta$: $F_1=\{z,u\}$,
  $F_2=\{x,u\}$ and $F_3=\{x,y\}$ pictured below. 

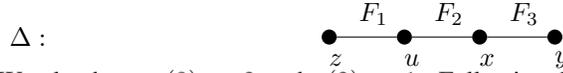
\begin{align*}
\begin{pgfpicture}
\pgfnodecircle{Node9}[virtual]{\pgfxy(-5.25, 0)}{0.1cm}
\pgfnodecircle{Node1}[fill]{\pgfxy(-1, 0)}{0.1cm}
\pgfnodecircle{Node2}[fill]{\pgfxy(0, 0)}{0.1cm}
\pgfnodecircle{Node3}[fill]{\pgfxy(1, 0)}{0.1cm}
\pgfnodecircle{Node4}[fill]{\pgfxy(2, 0)}{0.1cm}
\pgfnodeconnline{Node1}{Node2} \pgfnodeconnline{Node2}{Node3}
\pgfnodeconnline{Node3}{Node4} 
\pgfputat{\pgfxy(-1,-.3)}{\pgfbox[left,center]{$z$}}
\pgfputat{\pgfxy(0,-.3)}{\pgfbox[left,center]{$u$}}
\pgfputat{\pgfxy(1, -.3)}{\pgfbox[left,center]{$x$}}
\pgfputat{\pgfxy(2, -.3)}{\pgfbox[left,center]{$y$}}
\pgfputat{\pgfxy(-.6,.3)}{\pgfbox[left,center]{$F_1$}}
\pgfputat{\pgfxy(.4,.3)}{\pgfbox[left,center]{$F_2$}}
\pgfputat{\pgfxy(1.4, .3)}{\pgfbox[left,center]{$F_3$}}
\pgfputat{\pgfxy(-5.25, .0)}{\pgfbox[left,center]{$\Delta:$}}
\end{pgfpicture}
\end{align*}

We also have $\tau(3)=2$ and $\tau(2)=1$. Following \cref{construct} we obtain a tree $G$ with vertices indexed
by the monomials generating $I$ which supports a minimal free
resolution of $I$.

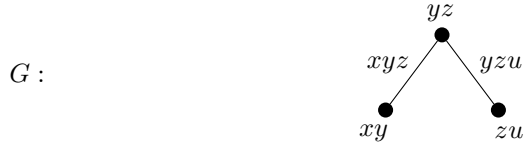
\begin{align*}
\begin{pgfpicture}
\pgfnodecircle{Node9}[virtual]{\pgfxy(3.25, .5)}{0.1cm}
\pgfnodecircle{Node1}[fill]{\pgfxy(8.25, 0)}{0.1cm}
\pgfnodecircle{Node2}[fill]{\pgfxy(9, 1)}{0.1cm}
\pgfnodecircle{Node3}[fill]{\pgfxy(9.75, 0)}{0.1cm}
\pgfnodeconnline{Node1}{Node2}
\pgfnodeconnline{Node2}{Node3}
\pgfputat{\pgfxy(7.9,-.3)}{\pgfbox[left,center]{$xy$}}
\pgfputat{\pgfxy(8,.6)}{\pgfbox[left,center]{$xyz$}}
\pgfputat{\pgfxy(8.8, 1.3)}{\pgfbox[left,center]{$yz$}}
\pgfputat{\pgfxy(9.5,.6)}{\pgfbox[left,center]{$yzu$}}
\pgfputat{\pgfxy(9.7, -.3)}{\pgfbox[left,center]{$zu$}}
\pgfputat{\pgfxy(3.25, .5)}{\pgfbox[left,center]{$G:$}}
\end{pgfpicture}
\end{align*}

\end{example}

\subsection{Cell complexes and cellular resolutions}

As a generalization of the Taylor resolution, one could homogenize the
cellular chain complexes of CW complexes to obtain free resolutions of
monomial ideals, which has the advantage that often there are fewer
faces in each dimension, meaning that the cellular resolution will 
potentially be closer to a minimal one. This idea was first developed by
Bayer and Sturmfels~\cite{BS} and has been expanded by many
authors since. Morse resolutions are an example of cellular
resolutions.

To formally define a finite regular CW complex (also
  known as a finite regular cell complex), we will use
  \cite{M80, OW}. Let $B^n$, with $n \geq 1$, denote the $n$-dimensional closed ball 
  $$
  B^n=\left\{(a_1, \dots, a_n)\in \mathbb R^n\mid \sum_{i=1}^n a_i^2\le 1\right\}\,.
  $$ The $n$-dimensional open ball is the interior $\text{int}({B^n})$
  of $B^n$. The $(n-1)$-dimensional sphere $S^{n-1}$ is the boundary
  of $B^n$. 

  A topological space is called a(n) {\bf (open) cell} of dimension $n$,
  or $n$-cell, if it is homeomorphic to $\text{int}(B^n)$  when $n\geq 1$ and to a point when $n=0$.
  A {\bf cell decomposition} of a space $X$ is a family $\Gamma=\{ c_i:i\in
  \mathcal I\}$ of pairwise disjoint subspaces of $X$ such that each
  $c_i$ is a cell and $X=\bigcup_{i\in \mathcal I}c_i$. We say that
  $\Gamma$ is finite when  the index set $\mathcal I$ is finite.

Given a cell decomposition as above and $n\ge 0$, let $\Gamma^n$
denote the set of all $n$-cells in $\Gamma$. The $n$-{\bf skeleton} of
$X$ is the subspace
$$
X^{n}=\bigcup_{0\le i\le n} \  \bigcup_{c\in \Gamma^i}c\,.
$$
The elements of the $0$-skeleton $X^0$ are called {\bf vertices}.

\begin{definition}
\label{d:cell}
A {\bf finite CW complex} is a Hausdorff space $X$, together with a
finite cell decomposition $\Gamma$ such that for each $n\ge 0$ and
$c\in \Gamma^n$ there is a continuous map $\Phi_c\colon B^n\to X$ that
restricts to a homeomorphism $\Phi_c\big{|}_{\text{int}(B^n)}\colon
\text{int}(B^n)\to c$ and takes $S^{n-1}$ into $X^{n-1}$. A finite CW
complex is also referred to as a finite {\bf cell complex}. When the
cell decomposition $\Gamma$ is understood (or implied), we will use
only the letter $X$ to refer to the CW complex.
\end{definition}

As described in~\cite[Lemma 2.2.6]{OW}, CW complexes can
also be defined by constructing the skeleton sets recursively, in
terms of a procedure of adjunction of cells of increasing dimensions,
starting with a discrete set of points as the $0$-skeleton.

\section{Acyclic matchings}\label{s:matchings}

In this section we set the foundations for the construction of a Morse
resolution of $I^r$ from an acyclic matching on the face poset of the
Taylor complex. We define the Morse complex itself in
\cref{mainresult}, give a detailed description of it in
\cref{t:critical-cells} and prove it supports a minimal resolution of
$I^r$ in \cref{t:minimality}.

We begin with some background in discrete Morse theory.

\begin{terminology}\label{amatching} 
Let $V$ be a finite set. We denote by $2^V$ the set of subsets of
$V$. Let $Y$ denote a subset of $2^V$. We call the elements $\sigma$
of $Y$ {\bf cells}. If $\sigma$ has exactly $n$
elements, then it is called an {\bf $(n-1)$-cell}. A cell with
only one element is a $0$-cell, and the empty set is a $(-1)$-cell.
 We note that the Taylor complex is a simplex consisting of the set of all subsets of a finite set $V$ of vertices. When convenient, we view the faces of the Taylor complex as cells of the $2^V$ set. 

We define the {\bf directed graph} $G_Y$ whose vertex
  set $\{\sigma \mid \sigma \in Y\}$ is  the set of cells of $Y$, and
with directed edges $E_Y$ consisting of $\sigma\to \sigma'$ with
$\sigma'\subseteq \sigma$ and $|\sigma|=|\sigma'|+1$.
A {\bf matching} of $G_Y$ is a set $A\subseteq E_Y$ of edges of $G_Y$
with the property that each cell of $Y$ occurs in at most one edge of
$A$.  A {\bf cycle} is a series of $n \geq 3$ directed edges $\sigma_0 \to \sigma_1 \to \cdots \to \sigma_n$ with $\sigma_0=\sigma_n$ and $\sigma_i \not= \sigma_j$ for $0 \leq i < j < n$.

For a matching $A$, let $G^A_Y$ be the graph whose edge set is
$$
E_Y^A=(E_Y\ssm A)\cup \{\sigma'\to \sigma \mid \sigma\to \sigma'\in A\}.
$$ Note the reversal of the direction of the edges in $A$.  In this
edge set, we think of the oriented edges in $E_Y\ssm A$ as pointing
down and the oriented edges $\sigma' \to \sigma$ with $\sigma \to
\sigma'$ in $A$ as pointing up.  In an abuse of language/notation, we
will simply say that the edges in $A$ \lq\lq point
up\rq\rq. Note that an edge that goes down corresponds
  to a decrease in the cardinality of the cells it connects, while an
  edge that points up corresponds to an increased cardinality. 
\end{terminology}

By definition, a cycle in $G_Y^A$ cannot have two consecutive edges in
$A$ since consecutive edges share a cell of $Y$. Since each edge of
$G_Y^A$ connects cells whose cardinalities differ by precisely one, a
cycle must contain the same number of edges that point down as edges
that point up. Combined, these two observations show that edges of a
cycle in $G_Y^A$ alternate between edges in $A$ and edges not in $A$.

The diagram in \cref{f:cyclefig} describes a cycle  in $G_Y^A$ for a matching $A$, 
where $r>1$ and for $i=1, \ldots, r$ with indices mod
  $r$ we have $\sigma_i\to \sigma'_{i+1}$ are edges in $E_Y\ssm A$ and
$\sigma_i\to \sigma'_{i}$ are edges in $A$, which are reversed
in $E_Y^{A}$ .  

\begin{figure}
\centerline{
\begin{tikzpicture}[label distance=-5pt] 
\coordinate (A) at (0, 0);
\coordinate (B) at (0, 1);
\coordinate (C) at (1, 0);
\coordinate (D) at (1,  1);
\coordinate (E) at (2, 0);
\coordinate (F) at (2,1);
\coordinate (G) at (3,0);
\coordinate (H) at (3, 1);
\coordinate (I) at (3.5, .5);
\coordinate (J) at (4.5, .5);
\coordinate (K) at (5, 0);
\coordinate (L) at (5, 1);
\coordinate (M) at (6, 0);
\coordinate (N) at (6,  1);
\coordinate (Z) at (3, -.5);
 \draw[black, fill=black] (A) circle(0.04);
\draw[black, fill=black] (B) circle(0.04);
 \draw[black, fill=black] (C) circle(0.04);
 \draw[black, fill=black] (D) circle(0.04);
 \draw[black, fill=black] (E) circle(0.04);
 \draw[black, fill=black] (F) circle(0.04);
 \draw[black, fill=black] (G) circle(0.04);
  \draw[black, fill=black] (H) circle(0.04);
  \draw[black, fill=black] (K) circle(0.04);
\draw[black, fill=black] (L) circle(0.04);
 \draw[black, fill=black] (M) circle(0.04);
 \draw[black, fill=black] (N) circle(0.04);
\draw[->] (A) -- (B);
\draw[->] (B) -- (C);
\draw[->] (C) -- (D);
\draw[->] (D) -- (E);
\draw[->] (E) -- (F);
\draw[->] (F) -- (G);
\draw[->] (G) -- (H);
\draw[-] (H) -- (I);
\draw[->] (J) -- (K);
\draw[->] (K) -- (L);
\draw[->] (L) -- (M);
\draw[->] (M) -- (N);
\draw[->] (N) -- (A);
\node[label = below :$\sigma'_0 $] at (A) {};
\node[label = above:$\sigma_0$] at (B) {};
\node[label =  below :$\sigma'_1$] at (C) {};
\node[label = above :$\sigma_1$] at (D) {};
\node[label = below :$\sigma'_2$] at (E) {};
\node[label = above:$\sigma_2$] at (F) {};
\node[label = below:$\sigma'_3$] at (G) {};
\node[label = above:$\sigma_3$] at (H) {};
\node[] at (3.75,0) {$\ldots$};
\node[] at (3.75,1) {$\ldots$};
\node[label = below :$\sigma'_{r-2} $] at (K) {};
\node[label = above:$\sigma_{r-2}$] at (L) {};
\node[label =  below :$\sigma'_{r-1}$] at (M) {};
\node[label = above :$\sigma_{r-1}$] at (N) {};
\end{tikzpicture}}
\caption{From \cite[page 54]{JJ} (with slightly different terminology)}\label{f:cyclefig}
\end{figure}
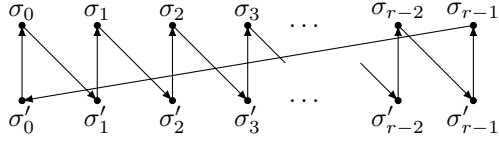

We say $A$ is {\bf acyclic} if $G^A_Y$ is acyclic (i.e. does not
contain directed cycles).  Given an acyclic matching $A$, the cells of
$Y$ that do not appear in the edges of the matching are called the
$A$-{\bf critical cells}; these are cells which are
unmatched.  The following lemma shows how acyclic matchings can be contracted or preserved in alternate ambient spaces.

\begin{lemma}\label{inclusions}
Let $Y\subseteq 2^V$. 
\begin{enumerate}[\quad\rm(1)] 
\item Let  $Y'\subseteq 2^V$ such that $Y\subseteq Y'$. If $A\subseteq E_Y$, then $A$ is an acyclic matching of $G_{Y'}$ if and only if $A$ is an acyclic matching of $G_Y$. 
\item If $A$ is an acyclic matching of $G_Y$ and $A'\subseteq A$, then $A'$ is an acyclic matching of $G_Y$ as well. 
\end{enumerate}
\end{lemma}

\begin{proof}  
(1) Assume $A$ is an acyclic matching on $G_Y$.  We show that $A$ is
  an acyclic matching on $G_{Y'}$. We see that $A$ is a matching on
  $G_{Y'}$ because the cells in $Y'\ssm Y$ do not occur in
  any of the edges of $A$, since $A\subseteq E_Y$. We need to show: If
  $G_Y^A$ is acyclic, then $G_{Y'}^A$ is acyclic. Thus, assume $G_Y^A$
  is acyclic. If there were a cycle in $G_{Y'}^A$ then it would involve
  at least one cell $\sigma\in Y'\ssm Y$. Since $A\subseteq
  E_Y$, there are no edges of $A$ that have $\sigma$ as a
  vertex. Therefore, the edges of the cycle that contain $\sigma$
  are both in $E_{Y'}\ssm A$, hence they both point
  down. However one cannot have in a cycle two consecutive edges that
  point down, as was previously observed; see \cref{f:cyclefig}. 

If $A$ is an acyclic matching on $G_{Y'}$, it is clear that it is also an acyclic matching on $G_Y$. 

(2) See for example~\cite[Lemma 3.2.3]{Ba}.
\end{proof}

 The next lemma will be used later as a source of acyclic matchings.

\begin{lemma}\label{matching-lemma}
Let $v\in V$. Then
\begin{equation*}
\begin{aligned}
A_Y^{v} = \big\{\sigma\to \sigma'\in E_Y\mid \, v\in \sigma   {\text{ and }} \sigma'=\sigma\ssm \{v\}\big\}.
\end{aligned}
\end{equation*}
is an acyclic matching on $G_Y$. 
\end{lemma}

\begin{proof}
Set $A=A_Y^{v} $.  To show that $A$ is a matching, assume $\sigma\in Y$
belongs to two edges in $A$. If the two edges are $\sigma\to \sigma'$
and $\sigma\to \sigma''$, then $\sigma'=\sigma''=\sigma\ssm
\{v\}$. If the two edges are $\sigma'\to \sigma$ and $\sigma''\to
\sigma$, then $\sigma'=\sigma''=\sigma\cup \{v\}$. If the edges are
$\sigma \to \sigma'$ and $\sigma''\to \sigma$, then we see that $v\in
\sigma$ because of the first edge and $v\notin\sigma$ because of the
second edge, a contradiction.

To show that the matching is acyclic, we need to show that $G_Y^{A}$
is acyclic. Assume there is a cycle in $G_Y^A$. Label it as in
\cref{f:cyclefig}. With the notation as in \cref{amatching}, for all
$i$ we have $\sigma'_i\subseteq \sigma_i$, $|\sigma_i|=|\sigma_{i+1}|$, and
$|\sigma'_i|=|\sigma'_{i+1}|$, where $|\sigma_i|=|\sigma'_i|+1$.  Since
$\sigma_i\to \sigma'_{i}\in A$, we see that $v\in \sigma_i$ and $v\notin
\sigma'_i$.

The cardinality of $\sigma'_1$ is one less the cardinality of $\sigma_0$, hence $\sigma_0=\sigma'_1\cup \{v'\}$ for some $v'\in V$. Since $v\in \sigma_0$, but $v\notin \sigma'_1$, we must have $v=v'$.  However, $\sigma'_1\cup \{v\}=\sigma_1$, and hence $\sigma_1=\sigma_0$,  a contradiction.\end{proof}

The following fact, which is a reformulation of Lemma 4.2 in
\cite{JJ}, is essential to \cref{mainresult}:

\begin{lemma}[{\bf Cluster Lemma} {\cite[Lemma 4.2]{JJ}}]\label{clusterlemma} 
Let $Y$ be a family of finite subsets of a set $V$.  Assume that there
exists a partition $Y=\bigcup_{q\in Q}Y_q$ indexed by a poset $Q$ with
the following property:
$$
\text{If $\sigma\in Y_q$ and $\sigma'\in Y_{q'}$ satisfy $\sigma' \subseteq \sigma$, then $q'\le q$.}
$$
Let $A_q$ be an acyclic matching of $G_{Y_q}$ for each $q$. Then the union $A=\bigcup_{q\in Q}A_q$ is an acyclic matching of $G_Y$. 
\end{lemma}

\begin{definition}\label{d:grading}  With notation as in
    \cref{amatching}, a {\bf grading} on $Y$ is an order-preserving
  function $\gr\colon Y\to P$, where $P$ is a poset, and $Y$ is
  ordered with respect to inclusion. If $Y$ is equipped with a
  grading, we say that an acyclic matching $A$ on $G_Y$ is {\bf
    homogeneous} provided that $\gr(\sigma)=\gr(\sigma')$ for all
  $\sigma\to \sigma'\in A$.
\end{definition}

 If $I$ is a monomial ideal and $X=\taylor(I)$ is the Taylor complex
 of $I$, and $P=\LCM(I)$ is the $\lcm$-lattice of $I$, then we can define
 a grading on $X$ via the function $$\lcm \colon X \to \LCM(I)$$
 where $\lcm(\sigma)=\lcm(m_i\mid i \in \sigma)$.  Note
   that this grading is exactly the monomial labeling of the faces of
   the Taylor complex, as in \cref{e:running}.

\begin{example}\label{e:running3} 
Let $I = (xy, yz, zu)$ be the ideal in \cref{e:running} and let $X$ be the
Taylor complex of $I$. The graph $G_X$ is shown below. The only
homogeneous acyclic matching in $G_X$  under the $\lcm$
  grading is the singleton edge $\{b\}$, since $\lcm(xy, yz, zu)=
\lcm(xy, zu)$.

\begin{equation*}
\xymatrix{ & & & \{xy, yz, zu\} \ar[dlll]_a \ar[d]^b \ar[drrr]^c & &
  \\ \{xy, yz\} \ar[dr]_d \ar[drrr]^e & & & \{xy, zu\} \ar[dll]_f
  \ar[drr]^g & & & \{yz, zu\} \ar[dlll]_h \ar[dl]^l \\ & \{xy\} & &
  \{yz\} & & \{zu\}}
\end{equation*}
\end{example}

Our motivation in what follows is the next statement, which is
a special case of \cite[Proposition~1.2, Theorem~1.3]{BW}.

\begin{theorem}[{\bf Resolutions from acyclic matchings}]\label{t:BW}
  If $I$ is a monomial ideal, $X$ is the Taylor complex of $I$
  graded by the $\lcm$ function, and $A$ is a homogeneous acyclic
  matching of $G_X$, then there is a CW complex $X_A$ which supports a
  multigraded free resolution of $I$. The $i$-cells of $X_A$ are in
  one-to-one correspondence with the $A$-critical $i$-cells of $X$.
  \end{theorem}

\begin{remark}
In~\cite{BW}, the authors define the notion of acyclic
matchings in the context of CW complexes. Since we will only apply the
results of \cite{BW} in the case of a simplex, we opted to give the
less general definitions presented here, as it is simpler to
work with them.

\end{remark}

\section{Powers of monomial ideals of projective dimension one}\label{s:ppd1}

We now turn our attention to the monomial labelings of the faces of
the Taylor complex of $I^r$ when $I$ is a square-free monomial ideal
of projective dimension~$1$ with a minimal resolution supported on a
tree $G$.  In other words, we investigate the monomials in the $\lcm$-lattice of $I^r$. The definitions and results established here 
provide the technical machinery used to define matchings and show they are homogeneous under the lcm labelings in 
subsequent sections.

In \cref{construct} we detailed the construction of the graph $G$ from
the (ordered) generators $m_1,\ldots,m_q$ of $I$.  In the last step of
\cref{construct}, note that since $F_i$ is a leaf of $\Delta_i$, there
is a free vertex $x \in F_i$ such that $x \notin F_j$ for all
$j<i$. Equivalently, $$x \notin F_i^c \mbox{ and } x\in F_j^c \mbox{
  for all } j<i. $$ Therefore for every $i \in \{1,\ldots,q\}$ there exists
a variable $x \in \{x_1, \ldots, x_n\}$ such that
\begin{align}\label{x}
  x \nmid   m_i \qand x\mid m_j \quad \mbox{ for all } \quad    j<i. 
\end{align}

We also notice that
$$\lcm(m_i,m_j)=\prod_{x_t\in F_i^c\cup F_j^c}x_t.$$ In particular, we
have $m_u\mid \lcm(m_i,m_j)$ if and only if $F_u^c\subseteq F_i^c\cup
F_j^c$, or equivalently, $F_i\cap F_j\subseteq F_u$.  In view of
\cref{construct}, since $F_i \cap F_j \subseteq
  F_{\tau(i)}$ whenever $j<i$, we have
$$\lcm(m_i,m_j, m_{\tau(i)})=\lcm(m_i,m_j)\quad\text{for all $j<i$},$$ and in particular,
\begin{align}\label{r:lcms}
  m_{\tau(i)}\mid \lcm(m_i,m_j)\quad\text{for all $j<i$}.
\end{align}

Let $\{m_1,\ldots,m_q\}$ be the square-free monomial
  generating set for $I$.  If $\ba=(a_1,\ldots,a_q)\in (\NN\cup \{0\})^q $, we set 

$$
|\ba|=a_1+\dots+a_q\quad\text{and}\quad \bma=m_1^{a_1}\cdots
m_q^{a_q}\,.
$$
For each $r\ge 0$, we further set 
$$\cN_r=\{\ba\in (\NN\cup \{0\})^q \mid |\ba| =r\}\quad\text{and}\quad \cM_r=\{\bma \mid \ba \in \cN_r\}\,.$$
Since $I$ is
generated by $m_1,\ldots,m_q$, the ideal $I^r$ is generated by
monomials in $\cM_r$. We show that
if $\pd_S(I)=1$, then $\cM_r$ is a minimal generating set for $I^r$.

\begin{proposition}[{\bf Uniqueness of generators}] \label{same-ems}
Assume $I$ is a square-free monomial ideal with $\pd_S(I)=1$ and let
$\{m_1,\ldots,m_q\}$ be the square-free monomial
  generating set for $I$. Suppose $\ba, \bb \in \cN_r$ for some $r
>0$. Then
$$\bma=\bmb \iff \ba=\bb.$$ In particular, $\cM_r$ forms a minimal
generating set for $I^r$.
\end{proposition}

\begin{proof} 
One direction of the statement is clear. For the other direction,
suppose without loss of generality the generators
  $m_1,\ldots,m_q$ are ordered as in \cref{construct}, $$\ba = (a_1,
\ldots, a_q) \neq \bb = (b_1, \ldots, b_q) \qand \bma=\bmb.$$ If we
set $$\ba\cap\bb=(\min(a_1,b_1),\ldots, \min(a_q,b_q)),$$ then by
cancelling $\bm^{\ba\cap\bb}$ from both sides of the equation
$\bma=\bmb$, we can assume without loss of generality that for all
$1 \leq j\leq q$, if $a_j \neq 0$ then $b_j=0$ and vice versa.

Let $k=\max\{j \mid a_j\neq 0 \mbox{ or } b_j \neq 0\}$, and suppose
without loss of generality that $a_k\neq 0$.  In particular, it follows
that $$b_j=0 \mbox{ for all } j \geq k.$$  

By \eqref{x}, there exists $x\in \{x_1, \dots, x_n\}$ such that
$$x \nmid m_k \mbox{ and } x\mid m_j \mbox{ for all } j<k.$$ So $x
\mid m_j$ for every $j$ such that $b_j \neq 0$, and therefore
$x^{|\bb|} \mid \bmb$. Since $|\ba|=|\bb|=r$, this implies that
$x^{|\ba|} \mid \bma$. But since $a_k \neq 0$ 
  and each $m_i$ is square-free then $x \mid m_k$, which is impossible.
\end{proof}

The definitions below will be used in the sequel.  It may
  be helpful to the reader to consider \cref{e:running2}, which
  illustrates some of the concepts.

\begin{definition}\label{d:generators}
   Let $I$ be a square-free monomial ideal of projective
    dimension $1$ and let $\{m_1,\ldots,m_q\}$ be the square-free
    monomial generators of $I$ ordered as in
    \cref{construct}. Suppose $\ba, \bb \in \cN_r$, where
$$\ba = (a_1, \ldots, a_q) \qand \bb = (b_1, \ldots, b_q).$$

\begin{itemize}
\item The {\bf support} of $\bma$, or of $\ba$, is the
  set $$\supp(\bma)=\supp(\ba)=\{j \mid a_j \neq 0\}
  \subseteq \{1, 2, \dots, q\}.$$

   \item Define $\prec$ by $\bmb \prec \bma$ or $\bb \prec
     \ba$ if  $b_j<a_j$ for the largest $j$ such that $b_j \neq a_j$.

   \item Define $\dom (\bma,\bmb)$ or
     $\dom (\ba,\bb)$ to be the largest index where $\ba$ and $\bb$
     differ when $\ba \not= \bb$.  In other words,
 $$\dom(\bma,\bmb)=\dom(\ba,\bb)=\max\{j \mid a_j
 \neq b_j\}.$$
 If $\ba=\bb$, set $\dom(\bma,\bmb)=\dom(\ba,\bb)= -\infty$.\\
 Since $|\ba|=|\bb|$, we note that $\dom(\ba,\bb)>1$ when $\ba\not= \bb$.
  \item Recall that a function $\tau\colon\{1, 2, \dots, q\}\to\{1, 2, \dots, q\}$ satisfying
    $\tau(1)=1$ and $\tau(i)<i$ for $i>1$ is defined in
    \cref{construct}.  If $j \in \supp(\ba)$,  and
      $\mathbf{e}_1, \ldots, \mathbf{e}_q$ denote the standard basis
      vectors for $\mathbb{R}^q$, we set
$$
\pi_j(\ba)=\ba + \mathbf{e}_{\tau(j)}- \mathbf{e}_j \qand \pi_j(\bma)=\bm^{\pi_j(\ba)}=\frac{\bma \cdot
    m_{\tau(j)}}{m_{j}}$$
  and we set
  \begin{align}\label{d:pi-notation}
    \bpi(\bma) =\{\pi_j(\bma)\mid j \in \supp(\bma)\} \cup \{\bma\}.
  \end{align}
  
\end{itemize}
\end{definition}

 Note that if $j\ne j'$, then $\pi_j(\ba) \ne \pi_{j'}(\ba)$ and therefore
$\pi_j(\bma)$ and $\pi_{j'}(\bma)$ are distinct.

\begin{example}\label{e:running2} If $I=(xy,yz,zu)$ is the ideal in
\cref{e:running}, then we label our generators as $$m_1=xy \quad
m_2=yz \quad m_3=zu.$$  The generators
of $I^2$ are uniquely written as $\bma$ where $\ba \in \cN_2$. The
monomial $m = xyzu \in I^2$, for example, corresponds uniquely to the
vector $(1,0,1)$.

Following the notation in \cref{e:running1}, since the joint of $F_3$
in $\Delta$ is $F_2$, we have $\tau(3)=2$ and thus,
\begin{align*}
\bpi(xyzu)= &\bpi(\bm^{(1,0,1)})
= \{ \bm^{(1,0,1)}, \bm^{(1,1,0)} \}
= \{m_1m_3, m_1m_2 \}
= \{xyzu, xy^2z\}.
\end{align*}
\end{example}  

 We close this section with two technical results which
  are necessary for our main result later in the paper.

\begin{lemma}\label{l:bpi} If $\ba \in \cN_r$,
 and  $1 <j<k \in \supp(\ba)$, then  $$\pi_k(\ba) \prec
 \pi_{j} (\ba) \prec \ba.$$
Moreover $\dom(\pi_j(\ba), \pi_k(\ba))=k$ and
  $\dom( \ba, \pi_j(\ba))=j$. 
\end{lemma}

 \begin{proof} Suppose $\ba=(a_1,\ldots,a_q)$,
  $\pi_j(\ba)=(b_1,\ldots,b_q)$, and $\pi_k(\ba)=(c_1,\ldots,c_q)$.
   We have
  $$b_i=a_i \mbox{ if } i >j; \, b_i=a_i=c_i \mbox{ if } i>k; \mbox{ and } c_k=a_k-1 < a_k=b_k \mbox{
    and } b_j=a_j-1 < a_j$$ which settles our claim.
 \end{proof}

\begin{proposition}\label{mainprop}
Let $\bma, \bmb \in \cM_r$ such that $\bb \prec \ba$. If
$k=\dom (\bma,\bmb)$ then
$$\lcm(\bma, \bmb,\pi_k(\bma))=\lcm(\bma, \bmb).$$
\end{proposition}

\begin{proof}
 The inequality $\bb = (b_1, \ldots, b_q) \prec (a_1, \ldots,
a_q) = \ba$ implies that $$a_k > b_k,
\ a_{k+1}=b_{k+1},\ \ldots,\ a_q=b_q.$$

We must show $$\pi_k(\bma)\mid \lcm(\bma, \bmb).$$ In particular, we need to show
that for every variable $x$ and any integer $s$ such that $x^s\mid
\pi_k(\bma)$,  either $x^s\mid \bma$ or $x^s\mid \bmb$.

 If $x$ does not divide $\m_{\tau(k)}$ and $x^s\mid \pi_k(\bma)$ we see from
 the equality $$\pi_k(\bma)m_k=\bma m_{\tau(k)}$$ that $x^s\mid \bma$. 

Assume now $x\mid m_{\tau(k)}$ and $s$ is maximal such that $x^s\mid
\pi_k(\bma)$.  If $x\mid m_k$, then $x^{s+1} \mid \pi_k(\bma)m_k$,
hence $x^{s+1}\mid \bma m_{\tau(k)}$. Since $m_{\tau(k)}$ is
square-free, we conclude that $ x^s \mid \bma$.  Assume now that $x$
does not divide $m_k$. By \eqref{r:lcms} , $m_{\tau(k)} \mid
\lcm(m_k,m_j)$ for all $j<k$. We conclude that $x \mid m_j$ for all
$j<k$. We claim that $x^{s}\mid \bmb$ in this case. 

Let $t$ denote the largest integer such that $x^{t}\mid
m_{k+1}^{a_{k+1}}\dots m_q^{a_q}$. As noted in \cref{d:generators}, necessarily $k>1$, and hence $\tau(k)<k$,  which implies that $m_{\tau(k)}\in
\{m_1, \dots, m_{k-1}\}$. Since $x$ does not divide $m_k$ but it
divides $m_1, \dots, m_{k-1}$ and all these monomials are square-free,
we have
$$ s=t+\sum_{i=1}^{k-1}a_i+1\,.
$$
Further, since $a_i=b_i$ for all $i>k$, we see that the largest integer
$s'$ such that $x^{s'}\mid \bmb$ is
$$ s'=t+\sum_{i=1}^{k-1}b_i\,.$$

To show $x^s\mid \bmb$ we need to show $s\le s'$. Assuming $s'<s$, we
have 
$$\sum_{i=1}^{k}b_i=\sum_{i=1}^{k-1}b_i+b_k \le
\sum_{i=1}^{k-1}a_i+b_k<\sum_{i=1}^{k-1}a_i+a_k=\sum_{i=1}^{k}a_i$$
The strict inequality above contradicts the fact that $|\ba|=|\bb|=r$.
\end{proof}

\section{$I^r$ has a resolution supported on a CW complex}\label{mainsection}

This section contains our main result. Recall that our setting is a
polynomial ring $S=\kk[x_1, \ldots, x_n]$ and an ideal $I$  with $\pd_S(I) =1$ generated by square-free monomials $\{m_1,\ldots, m_q\}$ in
$S$, ordered as in \cref{construct}. Given a positive integer $r$, we
denoted the generating set of $I^r$ by $\cM_r$, and showed that each element of
$\cM_r$ can be uniquely written as $\bma=m_1^{a_1}\cdots m_q^{a_q}$
where $\ba=(a_1,\ldots,a_q) \in (\NN \cup\{0\})^q$ is such that
$|\ba|=a_1+\cdots+a_q=r$ (\cref{same-ems}), in other words $\ba \in
\cN_r$. Let $X$ be the Taylor complex of $I^r$, with the $\lcm$ as the
grading function. Recall that $G_X$ denotes the directed graph on the set of faces of $X$, with edge set $E_X$; see \cref{amatching}.

Before stating our main theorem, we extend two functions, which appeared in \cref{d:generators}, to all faces
of $X$.

 \begin{definition} 
 \label{d:pi-cell}
   For $\sigma \in X$,
     $\sigma \neq \varnothing$, let ${\max}_\prec \sigma$ denote the
   largest monomial label of the vertices of $\sigma$ with respect to
   the order defined in \cref{d:generators}. Then if
   ${\max}_\prec \sigma = \bma$, define
$$
\dom(\sigma)= \begin{cases} 
	-\infty &  \sigma \subseteq \bpi(\bma) \\
	\max\{\dom(\bma,\bmb) \mid \bmb \in \sigma \setminus \bpi(\bma)\} & \sigma \not\subseteq \bpi(\bma).
	\end{cases}
 $$
 When $\dom(\sigma)\ne -\infty$ , we set 
$$\pi(\sigma)= \pi_{\dom(\sigma)}(\bma).$$
\end{definition}

Note that $\pi(\sigma)$ may not be in $\sigma$. Also, 
if $\ba \neq \bb$ and $\bma, \bmb \in I^r$, then
\begin{itemize}
\item  $1<\dom(\bma,\bmb)\le q$;
\item  (see also \cref{l:bpi}) for every $\bma \in
   \cM_r$ and $k \in \supp(\ba)$, $$\dom(\bma,\pi_k(\bma))=k.$$
 \end{itemize}

   \begin{example}\label{e:running4} In  our running 
     example $I =(xy, yz, zu)$, label the ordered generators as
     $m_1=xy$, $m_2=yz$ and $m_3=zu$. In view of Proposition \ref{same-ems}, we use the monomial labels of the vertices of the Taylor complex to represent these vertices. 
     If
     $$\sigma=\{xyzu, x^2y^2, y^2z^2,xy^2z\}=\{
     \bm^{(1,0,1)},\bm^{(2,0,0)},\bm^{(0,2,0)},\bm^{(1,1,0)}\},$$ which is a face of the Taylor complex of $I^2$, then
     $\max_\prec \sigma=\bm^{(1,0,1)}=xyzu$, and clearly
     $$\sigma \not\subseteq \bpi(\bm^{(1,0,1)})=\{\bm^{(1,0,1)},
     \bm^{(1,1,0)}\}$$ as calculated in \cref{e:running2}. So we have
     $$\sigma \ssm \bpi(\bm^{(1,0,1)}) =\{\bm^{(2,0,0)},\bm^{(0,2,0)}
     \}$$ and
     $$\dom(\bm^{(1,0,1)}, \bm^{(2,0,0)})=\dom(\bm^{(1,0,1)},
     \bm^{(0,2,0)})=3,$$ which results in $\dom(\sigma)=3$. Finally
     $\pi(\sigma)= \pi_3(\bm^{(1,0,1)})=\bm^{(1,1,0)}$, see \cref{e:running1}.
   \end{example}

 The main theorem of this section defines a matching $A$ and its critical cells
 using the sets $\bpi(\bma)$ for $\bma\in I^r$. To easily find the 
 non-empty critical cells of this matching, for each $\bma \in I^r$, identify
 all subsets of $\bpi(\bma)$ that contain $\bma$. These will be the
 critical cells. The remaining cells are matched in pairs that have
 the same largest monomial label. These pairs are  described using the  invariants in \cref{d:pi-cell}, which we note are dependent on the tree structure in \cref{construct}.  When confusion is not
 likely, for $\sigma \in X$ we will use the vertices of $\sigma$ and
 the labels of those vertices interchangeably.

\begin{theorem}[{\bf Main Theorem}]\label{mainresult} 
Let $I= (m_1, \dots, m_q)$ be a square-free monomial ideal, with $q$
generators, of projective dimension one in $S=\kk[x_1, \dots, x_n]$.
Let $r \in \NN$ and let $X$ be the Taylor complex of $I^r$. Let $A$ be
the following subset of $E_X$ 

\begin{equation}\label{d:A}
A=\Big\{\sigma\to (\sigma\setminus \{\pi(\sigma)\}) \mid \sigma \neq
\varnothing, \dom(\sigma)\ne -\infty, \pi(\sigma) \in
\sigma \Big\}.
\end{equation}

Then:
\begin{enumerate}[\quad\rm(1)]
\item The set $A$ is a homogeneous acyclic matching of $G_X$;
\item The set of critical cells for the matching $A$ consists of
 $$
 \{\varnothing \} \cup \bigcup_{\bma\in \cM_r}\big\{\sigma \cup \{\bma\}\mid \sigma \subseteq \bpi(\bma)\big\}\,;
 $$

\item There is a CW complex $X_A$  that supports a free
  resolution of $I^r$, and whose $i$-cells are in one-to-one
  correspondence with the $A$-critical $i$-cells of $X$.
\end{enumerate}
\end{theorem} 

The proof will be given in a series of steps.  For an illustration of the notation and concepts used in the proof, see
\cref{shadedtriangle}, at the end of the section, which details many of
the constructions herein.

\begin{proof}  
The proof consists of creating and refining a partition on the Taylor complex $X$ associated to $I^r$, and then describing carefully crafted matchings on each of the sets in the refined partition, which are then combined to produce the desired matching.

\noindent{\bf Step 1.} We construct a partition on $X$.

\noindent For each $\bma\in \cM_r$, set 
 $$X_{\bma} =\{\sigma \in X\mid {\max}_\prec \sigma=\bma\}.$$
 By \cref{l:bpi}, we have $\bpi(\bma) \in X_{\bma}$.
The sets $X_{\bma}$, together with $\{\varnothing\}$, form a partition of $X$. 
Notice that for $\bma, \bmb\in \cM_r$, $\sigma \in X_{\bma}$, and $\sigma'\in X_{\bmb}$ we have: 
\begin{equation}
\label{X-partition}
\sigma' \subseteq \sigma\Rightarrow \bmb\preceq \bma\,.
\end{equation}
To see this, suppose $\sigma'\subseteq \sigma$ for $\sigma \in X_{\bma}$ and $\sigma'\in X_{\bmb}$.  We have $\bmb \in \sigma' \subseteq \sigma$, hence  $\bmb \in \sigma$. Since $\sigma \in X_{\bma}$ and $\bmb \in \sigma$, we then have $\bmb\preceq \bma$. 

\medskip

\noindent{\bf Step 2.} Let $\bma\in \cM_r$. We construct a
partition of $X_{\bma}$. For $k\in \{2, \dots, q\}$ set  
\begin{align*}
Y_{\bma,k} = &\{\sigma \in X_{\bma}\mid \dom(\sigma)=k\}, \quad {\text{ and set }}\\
Y_{\bma,-\infty} = &\{\sigma \in X_{\bma}\mid \dom(\sigma)=-\infty\}
= \{\sigma\in X_{\bma}\mid \sigma\subseteq \bpi(\bma)\}\,.
\end{align*}
Note that $Y_{\bma,k}=\varnothing$ is possible. Also if $k\ne k'$,
then $Y_{\bma,k}$ and $Y_{\bma,k'}$ are disjoint. 
We have thus a partition of $X_{\bma}$ that consists of the
sets $Y_{\bma,k}$ that are not empty.  In
particular, we have
\begin{equation}
\label{Y-remain}
\bigcup_{k=2}^q Y_{\bma,k}=X_{\bma} \ssm Y_{\bma, -\infty}.
\end{equation}
Assume $\sigma \in Y_{\bma,k}$ and
$\sigma'\in Y_{\bma,k'}$ such that $\sigma'\subseteq \sigma$. By
definition, $\dom(\sigma)=k$ and $\dom(\sigma')=k'$. In particular, for
some $\bmb\in \sigma' \ssm \bpi(\bma)$, $\dom(\bma,\bmb)=k'$. Since
$\sigma' \subseteq \sigma$, we also have $\bmb \in \sigma \ssm
\bpi(\bma)$, and hence $k'\le \dom(\sigma)=k$. Thus for $k,k'\in \{2, \dots, q\}$, $\sigma \in Y_{\bma,k}$, and $\sigma'\in Y_{\bma,k'}$ we have 
\begin{equation}\label{Y-partition}
\sigma' \subseteq \sigma \Rightarrow  k'\le k.
\end{equation}

\medskip

\noindent{\bf Step 3.} Let $k\in \{2, \dots, n\}$ and $\bma\in
\cM_r$. We construct a homogeneous acyclic matching 
  (\cref{d:grading}) on $G_{Y_{\bma,k}}$ when $Y_{\bma,k}\neq
\varnothing$.

 If $Y_{\bma,k}\neq \varnothing$, we define a subset of $ E_{Y_{\bma,
     k}}$ as follows:
\begin{equation} \label{A}
\begin{aligned}
A_{\bma,k}=\Big\{ & \sigma\to \sigma'  \mid \sigma \in Y_{\bma,k}, \pi_k(\bma) \in \sigma, \sigma'=
\sigma \ssm \{\pi_k(\bma)\} \Big\}\,.
\end{aligned}
\end{equation}
\begin{claim} $A_{\bma,k} $ is a  homogeneous acyclic matching on $G_{Y_{\bma,k}}$ that
has no critical cells.
\end{claim}

\noindent{\it Proof of Claim.}  
Notice that if $\sigma \in Y_{\bma,k}$, then  $\sigma'=
\sigma \ssm \{\pi_k(\bma)\} \in Y_{\bma, k}$ as well. It follows directly from \cref{matching-lemma} that $A_{\bma,k}$ is an
acyclic matching.  To see that it is homogeneous, we need to see that
$$
\lcm(\sigma)=\lcm(\sigma') \quad\text{for all $\sigma\to \sigma'\in A_{\bma,k}$.}
$$ Since $\sigma \in Y_{\bma,k}$, we know by definition there is $\bmb\in
\sigma\ssm \bpi(\bma)$ such that $\dom(\bma,\bmb)=k$. By
\cref{mainprop}
$$\lcm(\bmb, \bma, \pi_k(\bma))=\lcm(\bmb,\bma),$$ hence, the matching is homogeneous.

Finally, let $\widetilde \sigma\in Y_{\bma,k}$. If $\pi_k(\bma)\in
\widetilde \sigma$, then we can take $\sigma=\widetilde\sigma$ in
\eqref{A}. Otherwise, we take $\sigma= \widetilde\sigma \cup
\{\pi_k(\bma)\}\in Y_{\bma, k}$.  In either case, $\widetilde\sigma$
is a vertex of an edge $\sigma\to \sigma'$ in $A_{\bma,k}$, hence the
matching $A_{\bma,k}$ has no critical cells.
\medskip

\noindent{\bf Step 4.} Let $\bma\in \cM_r$. We construct an acyclic matching on $G_{X_{\bma}}$ with set of critical cells equal to $Y_{\bma,-\infty}$.
 
In view of \eqref{Y-partition}, \cref{clusterlemma} gives that  the set 
$$A_{\bma}=\bigcup_{k=2}^q A_{\bma,k}$$
is a homogeneous acyclic matching on $G_{\bigcup_{k=2}^q
  Y_{\bma,k}}$.  Since each matching $A_{\bma,k}$ on $G_{Y_{\bma,k}}$ has no
critical cells, this matching has no critical cells either. 

Since
$\bigcup_{k=2}^q Y_{\bma,k}\subseteq X_{\bma}$, \cref{inclusions}(1) gives
that $A_{\bma}$ is also a homogeneous acyclic matching
on $G_{X_{\bma}}$. The set of critical cells is $X_{\bma}\ssm
\bigcup_{k=2}^q Y_{\bma,k}$, which is equal to $Y_{\bma, -\infty}$, by
\eqref{Y-remain}.

\medskip

\noindent{\bf Conclusion.} In view of
\eqref{X-partition}, \cref{clusterlemma} gives that
$\bigcup_{\bma\in \cM_r} A_{\bma}$ is a homogeneous acyclic matching on
$G_{\bigcup_{\bma\in \cM_r} X_{\bma}}$. Note that $\bigcup_{\bma\in \cM_r} A_{\bma}$
is precisely the set $A$ described in part (1) of the statement of the
theorem.

The set of critical cells is 
$$Y_{-\infty}=\bigcup_{\bma\in \cM_r} Y_{\bma,-\infty}=\{\sigma\in X\mid
\sigma\ne\varnothing, \, \sigma\subseteq \bpi({\max}_\prec\sigma)\}\,.$$
This is precisely the set whose cells are described in part (2) of the statement of the theorem. Since $$\bigcup_{\bma\in \cM_r} X_{\bma}=X\ssm \{\varnothing\}\subseteq
X\,,$$ \cref{inclusions}(1) implies that $\bigcup_{\bma\in \cM_r} A_{\bma}$ is an
acyclic matching on $G_X$; the set of critical cells is
$Y_{-\infty}\cup\{\varnothing\}$.

By \cite[Theorem~1.3]{BW} there is a CW complex $X_A$
  (the Morse complex of $X$) whose $i$-cells are in one-to-one
  correspondence with the $A$-critical $i$-cells of $X$, and $X_A$
  supports a free resolution of $I^r$.
\end{proof}

The following example illustrates many of the notations and concepts
used above. 

\begin{example}  \label{shadedtriangle}  In the running example $I =(xy, yz, zu)$, label the generators as before, namely 
 $$m_1=xy \quad m_2=yz \quad m_3=zu.$$ 
By \cref{same-ems}, the generators of $I^2$ are uniquely written as
$\bma \mbox{ where } \ba \in \cN_2$.

The Taylor complex $X$ for $I^2$ is a $5$-dimensional simplex whose
vertices are labeled with the $6$ generators of $I^2$. The monomial $m
= xyzu \in I^2$ corresponds to the vertex $(1,0,1)$.  The set $X_m$
from Step~1 of \cref{mainresult} consists of all faces of
  the Taylor complex containing $(1,0,1)$ and  any subset of vertices satisfying the constraints $$(a,b,c)\prec
  (1,0,1) \mbox{ and } a+b+c=2,$$ which are precisely the vertices
  $(2,0,0),(0,2,0),(1,1,0)$. Thus using the same monomial labeling convention as in Example \ref{e:running4}, we get
\begin{align*}
X_{xyzu} = &X_{\bm^{(1,0,1)}}\\
= &  \big\{ \{\bm^{(1,0,1)}\} \cup U \mid 
U \subseteq \{\bm^{(2,0,0)},\bm^{(0,2,0)},\bm^{(1,1,0)}\} \big\}\\
= &  \big\{ 
\{xyzu, x^2y^2, xy^2z, y^2z^2\}, 
\{xyzu, x^2y^2, xy^2z\},  
\{xyzu, x^2y^2, y^2z^2 \}, \\
&\{xyzu, xy^2z, y^2z^2\}, 
\{xyzu, x^2y^2\}, 
\{xyzu, xy^2z\}, 
\{xyzu, y^2z^2\}, 
\{xyzu\} \big\}.
\end{align*}

Next, we calculate the sets $Y_{m,k}$ for various $k$, as per Step~2
of \cref{mainresult}.  From \cref{e:running2} we know
$$\bpi(xyzu)= \{ \bm^{(1,0,1)}, \bm^{(1,1,0)} \}=\{xyzu, xy^2z\}.$$

For each $\sigma \in X_{xyzu}$, ${\max}_\prec \sigma=\bm^{(1,0,1)}$.
If $$\sigma=\{
\bm^{(1,0,1)},\bm^{(2,0,0)},\bm^{(0,2,0)},\bm^{(1,1,0)}\},$$
then clearly $\sigma \not\subseteq \bpi({\max}_\prec \sigma)$. For this $\sigma$,
$\dom(\sigma)=3$. 
In fact, $\dom(\sigma)=3$ for
any $\sigma\in X_{xyzu}$ as long as 
$$\sigma \not \subseteq \bpi({\max}_\prec \sigma) = \{ \bm^{(1,0,1)},
\bm^{(1,1,0)} \}.$$ Therefore,
\begin{align*}
Y_{xyzu, 3}= &  Y_{\bm^{(1,0,1)}, 3}\\
 = &  \big\{ 
    \{\bm^{(1,0,1)}, \bm^{(2,0,0)},\bm^{(0,2,0)},\bm^{(1,1,0)}\},\\
 &  \{\bm^{(1,0,1)}, \bm^{(2,0,0)},\bm^{(1,1,0)}\},
   \{\bm^{(1,0,1)}, \bm^{(2,0,0)},\bm^{(0,2,0)}\},\\
 & \{\bm^{(1,0,1)}, \bm^{(0,2,0)},\bm^{(1,1,0)}\},
   \{\bm^{(1,0,1)}, \bm^{(2,0,0)}\},
   \{\bm^{(1,0,1)}, \bm^{(0,2,0)}\}
\big\}.\\
\end{align*}
In other words, 
\begin{align*}
Y_{xyzu, 3}= &\big\{ \{xyzu, x^2y^2, xy^2z, y^2z^2\}, 
\{xyzu, x^2y^2, xy^2z\},\{xyzu, x^2y^2, y^2z^2\},
\\
& \{xyzu, xy^2z, y^2z^2\}, 
\{xyzu, x^2y^2\}, 
\{xyzu, y^2z^2\} 
\big\}.
\end{align*}

Viewed another way, $Y_{xyzu, 3} = X_{xyzu} - \left\{ \{xyzu, xy^2z\},
\{xyzu\} \right\}$.  Note that 
$$Y_{xyzu, 2}= \varnothing$$
and
$$Y_{xyzu, -\infty} = \left \{ \{\bm^{(1,0,1)}, \bm^{(1,1,0)}\}, \{
\bm^{(1,0,1)}\} \right\}=\left\{\{xyzu, xy^2z\}, \{xyzu\} \right\}. $$

Hence, recalling that $ \pi_3(\bm^{(1,0,1)})=\bm^{(1,1,0)}$ (see \cref{e:running1}),  the elements in $A_{m,3}$ consist of the following directed edges:

$$\begin{array}{ccc}
\{\bm^{(1,0,1)}, \bm^{(2,0,0)},\bm^{(0,2,0)},\bm^{(1,1,0)}\} &
   \longrightarrow &   \{\bm^{(1,0,1)}, \bm^{(2,0,0)},\bm^{(0,2,0)}\}\\
\{\bm^{(1,0,1)}, \bm^{(2,0,0)},\bm^{(1,1,0)}\} &
  \longrightarrow &  \{\bm^{(1,0,1)}, \bm^{(2,0,0)}\}\\
\{\bm^{(1,0,1)}, \bm^{(0,2,0)},\bm^{(1,1,0)}\} &
  \longrightarrow & \{\bm^{(1,0,1)}, \bm^{(0,2,0)}\}.
\end{array}
$$

There are $6$ critical cells in $ \cM_2$ which correspond to 
 \cref{f:critical} listed in \cref{e:list}. Note that $\{\bm^{(1,0,1)},
\bm^{(0,2,0)} \}$ is not a critical cell; i.e., it is not an edge in the diagram. However, $\{\bm^{(1,0,1)}, \bm^{(2,0,0)},\bm^{(0,2,0)}\}$ is a critical cell. Since the complex that supports the resolution does not contain  $\{\bm^{(1,0,1)},
\bm^{(0,2,0)} \}$ then it is not a simplicial complex, and hence the resolution is not simplicial.

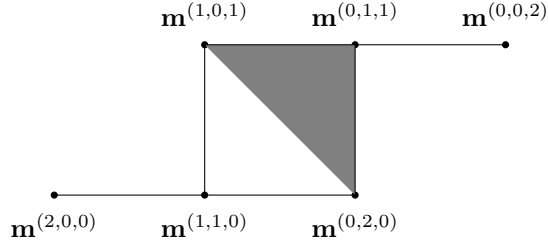
\begin{figure}
\begin{tikzpicture}
\coordinate (A) at (0, 0);
\coordinate (B) at (2, 0);
\coordinate (C) at (4, 0);
\coordinate (D) at (4,2 );
\coordinate (E) at (2, 2);
\coordinate (F) at (6, 2);
 \draw[black, fill=black] (A) circle(0.04);
 \draw[black, fill=black] (B) circle(0.04);
 \draw[black, fill=black] (C) circle(0.04);
 \draw[black, fill=black] (D) circle(0.04);
 \draw[black, fill=black] (E) circle(0.04);
 \draw[black, fill=black] (F) circle(0.04);
 \draw[draw = none, fill=gray] (E) -- (C) -- (D)  -- cycle;
\draw[-] (A) -- (B);
\draw[-] (D) -- (F);
\draw[-] (B) -- (C);
\draw[-] (B) -- (E);
\draw[-] (C) -- (D);
\draw[-] (E) -- (D);
\node[label = below :$\bm^{(2,0,0)}$] at (A) {};
\node[label = below :$\bm^{(1,1,0)}$] at (B) {};
\node[label = below :$\bm^{(0,2,0)}$] at (C) {};
\node[label = above :$\bm^{(0,1,1)}$] at (D) {};
\node[label = above :$\bm^{(1,0,1)}$] at (E) {};
\node[label = above :$\bm^{(0,0,2)}$] at (F) {};

\end{tikzpicture}
\caption{The critical cells for $I^2$ where $I=(xy,yz,zu)$}\label{f:critical}
\end{figure}

\end{example}

\section{The Morse complex $X_A$}\label{s:morse}

\cref{mainresult} states that there is a CW complex $X_A$ supporting a
free resolution of $I^r$, where $X$ is the Taylor complex of $I$ and
$A$ is an acyclic matching on the poset graph $G_X$ of $X$. Moreover,
the $i$-cells of $X_A$ are in one-to-one correspondence with the
$A$-critical $i$-cells of $X$.  Following the notation in~\cite{BW},
if $\sigma$ is an $A$-critical cell of $X$, we denote by $\sigma_A$
the unique corresponding cell of $X_A$. We use the notation
$\sigma'_A\le \sigma_A$ to say that the cell $\sigma'_A$ is contained
in the closure of the cell $\sigma_A$.

 Our goal in this section is to determine, given $A$-critical cells
 $\sigma$ and $\sigma'$ of $X$, under what conditions on $\sigma$ and
 $\sigma'$ do we get $$\sigma'_A\le \sigma_A?$$ It is not always
 immediately clear from the definition of a closure which cells are
 contained in the closure of other cells. Batzies and
 Welker~\cite{BW} characterized the cell ordering $\sigma'_A\le
 \sigma_A$ in the Morse complex $X_A$ in terms of certain paths in the
 directed graph $G^A_X$, called ``gradient paths'' (see
 \eqref{e:cells}).

 In this section, we will focus on the structure of the gradient paths
 in our setting, and show, in \cref{t:critical-cells}, exactly what
 the cell order $\sigma'_A\le \sigma_A$ in the Morse complex $X_A$
 means in terms of the $A$-critical cells $\sigma$ and $\sigma'$ of
 $X$. Given the technical nature of the discussions and the fact that
 they apply only to this particular proof, we have chosen to state
 \cref{t:critical-cells} early on. What follows after are all the
 components that go into its proof.

In the statement of \cref{t:critical-cells}, we denote the
$A$-critical $i$-cells of $X$ (as described in \cref{mainresult}(2))
by $$\sigma(\bma,D)= \{\bma\} \cup \{ \pi_j(\bma) \mid j \in D \}
\subseteq \bpi(\bma)$$ where $\ba\in \cN_r$, $D\subseteq
\supp(\ba)\ssm\{1\}$ and $|D|=i$.  

  We are now ready to state the main result of this section.
Note that
    part~(1) of \cref{t:critical-cells} follows immediately
    from \cref{mainresult}, but we have included it in order to have a
    complete statement for the characterization of the cells of
    $X_A$.

\begin{theorem}[{\bf The cells of $X_A$}]\label{t:critical-cells}
  Let $X_A$ be the Morse complex of the matching $A$ on the Taylor
  complex $X$ as described in \eqref{d:A}, and let $i>0$. Then:

 \begin{enumerate} 
\item (\cref{mainresult}) For every $i$-cell $c$ of $X_A$ there is a 
unique $A$-critical cell $\sigma=\sigma(\bma,D)$ of $X$ where $\bma
\in \cM_r$, $D \subseteq \supp(\bma) \ssm \{1\}$ and $|D|=i$, such
that $$c=\sigma_A.$$

 \item  If $c'$ is an $(i-1)$-cell of $X_A$, then
  $c' \leq c$ if and only if $c'=\sigma'_A$ where $$
    \sigma'=\sigma(\bma, D\ssm \{k\}) \qor \sigma'=\sigma(\pi_k(\bma),
    D\ssm \{k\})$$ for some $k\in D$. 
    \end{enumerate}
\end{theorem}

 \cref{t:critical-cells} gives a concrete description of
  the ordering of cells in the Morse complex (see~\eqref{e:cells}) in
  terms of the $A$-critical cells of the Taylor complex.  To start our
  way towards the proof, we further develop the notation used for
  critical cells.
\begin{notation}\label{n:critical-cells} 
 Let $\ba \in \cN_r$ and let $D=\{i_1, \dots, i_s\}$ be a subset of $\supp(\ba)$. Set 
$$
\pi_D(\ba)=\ba + \sum_{i\in D} \mathbf{e}_{\tau(i)}-\sum_{i\in D} \mathbf{e}_{i} \quad\text{and}$$
$$\pi_D(\bma)=\bm^{\pi_D(\ba)}=\frac{\bma m_{\tau(i_1)}m_{\tau(i_2)}\cdots
   m_{\tau(i_s)}}{m_{i_1}m_{i_2}\dots m_{i_s}}.$$
 When $D=\varnothing$ we set $\pi_{\varnothing}(\ba)=\ba$ and $\pi_{\varnothing}(\bma)=\bma$. Further, we set
    $$\sigma(\bma,D)= \{\bma\} \cup \{ \pi_i(\bma) \mid i \in D \}
 \subseteq \bpi(\bma) \qand \,$$
 $$\osigma(\bma, D)= \{\bm^{\pi_L(\ba)} \mid L\subseteq D\}.$$
  \end{notation}

Since $\pi_{\varnothing}(\bma)=\bma$, we have $\bma\in \osigma(\bma, D)$ for all $D$. Also, note that $\sigma(\bma, \varnothing)=\osigma(\bma, \varnothing)=\{\bma\}$. 

Observe that $\sigma(\bma,D) \subseteq \osigma(\bma,D)$, and if
$D=D'\cup D''$ with $ D'\cap D''=\varnothing$, then 
 \begin{equation}\label{e:D'}
   \pi_D(\bma)=\pi_{D'}(\pi_{D''}(\bma))=\pi_{D''}(\pi_{D'}(\bma)).
 \end{equation}

 Note that the $A$-critical $i$-cells of $X$ appearing in
 \cref{mainresult}(2) are precisely the cells $\sigma(\bma, D)$ with
 $\ba\in \cN_r$, $D\subseteq \supp(\ba)\ssm\{1\}$  and $|D|=i$.

 We now describe the cell ordering in the Morse complex
 in terms of gradient paths, following the authors in~\cite{BW}.  A
 {\bf gradient path} in the graph $G^A_X$ (defined in
 \cref{amatching}) is a directed path
$$\mathcal P \colon \sigma_0\to \dots\to \sigma_n$$ where $\sigma_0$ is the {\bf initial
  point} and $\sigma_n$ is the {\bf end point}; see, e.g.,
\cite[p.~165]{BW}.  For cells $\sigma, \sigma'$ of $X$, the set of all
gradient paths in $G^A_X$ with initial point $\sigma$ and end point
$\sigma'$ is denoted by $\text{GradPath}_A(\sigma,\sigma')$.

By~\cite[Proposition~7.3]{BW} if $\sigma'', \sigma$ are $A$-critical
cells of $X$ of dimensions $i-1$ and $i$, respectively, with 
$\sigma''_A$ and $\sigma_A$ the corresponding cells in $X_A$ of
dimensions $i-1$ and $i$, respectively, then

\begin{equation}\label{e:cells}
\sigma''_A \leq \sigma_A \iff \begin{cases}
  \sigma'' \subseteq \sigma \mbox{ or }\\
  \text{GradPath}_A(\sigma',\sigma'') \neq \varnothing \mbox{ for some } 
  \ \sigma' \subseteq \sigma \mbox{ with }  \dim(\sigma')=i-1.
\end{cases}
\end{equation}

\begin{discussion}[{\bf Gradient paths}]\label{d:grad} A gradient path
  between two $(i-1)$-cells  $\sigma'$ and $\sigma''$, where $\sigma''$ is $A$-critical, can be visualized in
  terms of edges pointing up or pointing down:

      $$\begin{array}{ccccccccccc}
    && \sigma_1 &&&&&& \sigma_{u-1} &&  \\
    &\nearrow &&\searrow &&&&\nearrow &&\searrow & \\ 
    \sigma'=\sigma_0&&&& \sigma_2 &\ldots&\sigma_{u-2}&&&& \sigma_{u}=\sigma'' \\ 
    \end{array}.$$

\bigskip
    
To see this, recall (\cref{amatching}) that an edge 
    points down if it corresponds to an inclusion between 
     cells  and an edge points up if it is 
      the reverse of an edge in $A$. Therefore there
    are no edges that point up and end with $\sigma''$ since $\sigma''$ is an $A$-critical cell, so $\sigma''$
    can only be reached via a down arrow.  Also, recall that one
    cannot have two consecutive edges pointing up, because every cell  appears only once in the matching A. So every up
      arrow must be followed by a down arrow and once a (down) arrow
    hits a critical cell, then the gradient path must stop, since
    there is no up arrow from a critical cell.

Our discussion applied to the  gradient path drawn above
  forces the following: \begin{itemize}
  \item $\sigma'=\sigma_0, \ldots, \sigma_{u-1}$ are not  critical
    cells since they are  included in edges pointing up.

  \item $\sigma_u=\sigma''$ is the first critical cell along this path,
    and the path stops here.

  \item The dimensions of the  cells along the path are $i-1,
    i,i-1,i,\ldots,i-1$, since the path is a series of up (dimension
    goes up by one) and down (dimension goes down by one) arrows.

    \item   An upward arrow  {\tiny
      $\begin{array}{ccc}
         &&\sigma_{j+1}\\
         &\nearrow &\\
         \sigma_{j}&&\\
       \end{array}$}
     indicates 
      \begin{equation}\label{e:upward}
        \sigma_{j+1}= \sigma_j \cup
      \{\pi_k(\bmb)\},
      \end{equation}
      where $\bmb=\max(\sigma_j)$ and
      $k=\dom(\sigma_{j})$ by Theorem~\ref{mainresult}. In particular
      $\max(\sigma_j)=\max(\sigma_{j+1})$ in this case.

  \item A downward arrow   {\tiny 
      $\begin{array}{ccc}
         \sigma_{j}&&\\
         &\searrow&\\
         & &\sigma_{j+1}\\
       \end{array}$} 
    indicates $\sigma_{j+1} \subseteq
    \sigma_j$ and $\max(\sigma_{j+1}) \preceq
    \max(\sigma_j)$.
    
    \item The previous two items show that $\max(\sigma_j)$ does not
      increase as one proceeds through the gradient path.
    \end{itemize}

  \end{discussion}

 We now start proving part (2) of
  \cref{t:critical-cells}.  We first describe the cells of $X_A$ and
the order relation on the cells, in terms of the faces of $X$ and possible gradient paths between
them. \cref{l:destination}
establishes that there is always a gradient path of the type needed
for \cref{t:critical-cells}, and \cref{l:grad} says that those are the
only gradient paths.

 \begin{lemma}\label{l:destination} 
   Let $\bma \in \cM_r$ and $\sigma=\sigma(\bma,D)$ for some
   $D \subseteq \supp(\bma) \ssm \{1\}$ with $|D| \geq 2$, and let
   $k \in D$.  Then there is a gradient path in $G^A_X$ from
   $\sigma \ssm \{\bma\}$ to $\sigma(\pi_k(\bma), D \ssm \{k\})$.
 \end{lemma}

 \begin{proof}
 Let $\sigma'=\sigma \ssm \{\bma\}$, $\sigma''=\sigma(\pi_k(\bma), D
 \ssm \{k\})$ and $D=\{d_1,\ldots,d_s\}$ where $$d_1 < \cdots <d_s.$$
 Then by \cref{l:bpi} we have  $$\pi_{d_s} (\bma) \prec
   \pi_{d_{s-1}}(\bma) \prec \cdots \prec \pi_{d_2}(\bma) \prec
   \pi_{d_1}(\bma).$$

   We start by constructing a gradient path in the most basic case
   $k=d_1=\min(D)$ carefully.  Following the rules in \cref{d:grad} we
   build the following gradient path in $G^A_X$, where the arrow from
   the noncritical cell  $\sigma'$ must be an up arrow, and the arrow
   to the critical cell  $\sigma''$ must be a down arrow. We use $u$ to denote the number of steps in the gradient path. 
     
  $$\begin{array}{ccccccccccc}
    && \sigma_1 &&&&&& \sigma_{u-1} &&  \\
    &\nearrow &&\searrow &&&&\nearrow &&\searrow & \\ 
    \sigma'=\sigma_0&&&& \sigma_2 &\ldots&\sigma_{u-2}&&&& \sigma_{u}=\sigma'' \\ 
     \end{array}.$$
   \medskip
   
Using \cref{l:bpi} we observe that $\max(\sigma_0)=\pi_{d_1}(\bma)$ and
   $\dom(\sigma_0)=d_s$, so by \eqref{e:upward}, we have only one choice for the upward arrows. Thus 

      \medskip

   \noindent {\bf $\sigma_1$: }  $\sigma_{1}= \sigma_0 \cup
   \{\pi_{d_s}(\pi_{d_1}(\bma))\}=\{\pi_{d_1} (\bma), \cdots,
   \pi_{d_s}(\bma), \pi_{\{d_s,d_1\}}(\bma)\}.$

   \medskip
   
   \noindent {\bf $\sigma_2$: } The next downward arrow is a simple
   elimination, and here we choose the option $$\sigma_2=\sigma_1 \ssm
   \{\pi_{d_s} (\bma)\}=\{\pi_{d_1}(\bma), \cdots,
   \pi_{d_{s-1}}(\bma), \pi_{\{d_s,d_1\}}(\bma)\}.$$

        \noindent {\bf $\sigma_3$: } Once again, since
     $\max(\sigma_2)=\pi_{d_1}(\bma)$ and $\dom(\sigma_0)=d_{s-1}$, 
     \begin{align*}
       \sigma_{3}=& \sigma_2\cup\{\pi_{d_{s-1}}(\pi_{d_1}(\bma))\}\\
       =& \{\pi_{d_1} (\bma), \cdots, \pi_{d_{s-1}}(\bma),
            \pi_{\{d_s,d_1\}}(\bma),
            \pi_{\{d_{s-1},d_1\}}(\bma)\}.
     \end{align*}
  
   We continue in this manner, with
   every downward arrow eliminating the largest remaining $d_i$, and
   finally we arrive at

   \medskip
   
   \noindent {\bf $\sigma_{u-1}$: } $\sigma_{u-1}= \{\pi_{d_1} (\bma),
     \pi_{d_2}(\bma),
      \pi_{\{d_s,d_1\}}(\bma), \cdots, \pi_{\{d_2,d_1\}}(\bma)\}$

         \medskip

     \noindent {\bf $\sigma_u$: } $\sigma_{u}=\sigma_{u-1}\ssm
       \pi_{d_2}(\bma) = \sigma(\pi_{d_1} (\bma), D \ssm \{d_1\})
       = \sigma''.$

         \medskip

   So we have shown the existence of the
   gradient path when $k=d_1$.
   
\smallskip

If $k=d_e \in D$ and $1<e<s$, we construct a gradient path from
   $\sigma'$ to $\sigma''$ below. For a cleaner picture we use the
   product of indices in $L$ to denote the monomial $\pi_L(\bma)$ for
   $L \subseteq D$, and we keep track of all these indices in the table below.
   We start with $\sigma_0=\{\pi_{d_1}
     (\bma), \cdots, \pi_{d_s}(\bma)\}.$

   \medskip
   
   {\tiny
   $$\begin{array}{c|c|c|c|c|c|c|c|c|c|c}

       \mbox{\bf Row } & \mbox{\bf Add} & & \mbox{\bf Add} &&&
       \mbox{\bf Add} && \mbox{\bf Add} & & \mbox{\bf Remaining }
       \pi_L(\bma)\\
      & & \mbox{\bf Delete} & & \mbox{\bf Delete} & & & \mbox{\bf
         Delete} & & \mbox{\bf Delete} &\\
       \hline
       &&&&&&&&&& \\
       1 &d_1d_s& &d_1d_{s-1}&&\cdots
       &d_1d_{e+1}&&d_1d_e&& d_2,\ldots,d_e,\\
         &&d_s&&d_{s-1}&&&d_{e+1}&&d_1& d_1d_e,\\
         &&&&&&&&&& d_1d_{e+1},\ldots, d_1d_s\\
       &&&&&&&&&& \\
       \hline
       &&&&&&&&&& \\
       2 & d_2d_s&&d_2d_{s-1}&&\cdots&d_2d_{e+1}&&d_2d_e&&  d_3,\ldots,d_e,\\
       &&d_1d_s&&d_1d_{s-1}&&&d_1d_{e+1}&&d_2&d_1d_e, d_2d_e,\\
       &&&&&&&&&&  d_2d_{e+1},\ldots,d_2d_s\\
       &&&&&&&&&& \\
       \vdots &&&&&\vdots&&&&& \vdots \\
       &&&&&&&&&& \\
       e-1 & d_{e-1}d_s&&d_{e-1}d_{s-1}&&\cdots&d_{e-1}d_{e+1}&&d_{e-1}d_e&&  d_e\\
       &&d_{e-2}d_s&&d_{e-2}d_{s-1}&&&d_{e-2}d_{e+1}&&d_{e-1}&d_1d_e,\ldots,d_{e-1}d_e, \\
       &&&&&&&&&&  d_{e-1}d_{e+1},\ldots,d_{e-1}{d_s}\\
       &&&&&&&&&& \\
       \hline
       &&&&&&&&&& \\
       e &d_ed_s&&d_ed_{s-1}&&\cdots&d_ed_{e+1}&&&& d_e\\
       &&d_{e-1}d_s&&d_{e-1}d_{s-1}&&&d_{e-1}d_{e+1}&&&d_1d_e,\ldots,d_{e-1}d_e, \\
       &&&&&&&&&&d_{e}d_{e+1},\ldots,d_{e}d_s \\
   \end{array}$$
}

     \medskip

     \noindent {\bf Row 1:} Recall that each up arrow adds a monomial
     and each down arrow deletes one, so the first series of arrows in
     the gradient path will start from $\sigma_0$ and do the following
     sequence of additions and deletions:
     \begin{align*}
       \mbox{ add } \pi_{\{d_1,d_s\}}(\bma),
       \mbox{ delete } \pi_{d_s}(\bma),
       &\mbox{ add } \pi_{\{d_1,d_{s-1}\}}(\bma),
       \mbox{ delete }\pi_{d_{s-1}}(\bma),\\
       \ldots,
       \mbox{ add } \pi_{\{d_1,d_{e+1}\}}(\bma),
       \mbox{ delete }\pi_{d_{e+1}}(\bma),
     &\mbox{ add } \pi_{\{d_1,d_e\}}(\bma),\mbox{\bf delete }\pi_{d_1}(\bma).
     \end{align*}
    The final down arrow in this row eliminates $\pi_{d_1}(\bma)$,
    making the monomial $\pi_{d_2}(\bma)$ the largest one in the
    remaining cell , whose elements are the monomials $\pi_L(\bma)$
    where $L$ ranges over index sets listed at the end of  Row~1.
     \medskip
     
     \noindent {\bf Middle rows:} In Row~$i$, $1<i<e$, by \cref{l:bpi}
     $\pi_{d_i}(\bma)$ is the largest remaining monomial at that spot
     in the path, and we go through the same moves:
     \begin{align*}
       \mbox{ add }\pi_{\{d_i,d_s\}}(\bma),
       \mbox{ delete }\pi_{\{d_{i-1},d_s\}}(\bma),
       &\mbox{ add }\pi_{\{d_i,d_{s-1}\}}(\bma),
       \mbox{ delete }\pi_{\{d_{i-1},d_{s-1}\}}(\bma), \\
       \ldots, 
       \mbox{ add }\pi_{\{d_i,d_{e+1}\}}(\bma),
       \mbox{ delete }\pi_{\{d_{i-1},d_{e+1}\}}(\bma), 
       & \mbox{ add }\pi_{\{d_i,d_e\}}(\bma),
      \mbox{\bf delete }\pi_{d_i}(\bma),
     \end{align*}
     eliminating $\pi_{d_i}(\bma)$ at the very end to make $\pi_{d_{i+1}}(\bma)$ the largest monomial.

  \medskip
     
     \noindent {\bf Row e:} The final row of the table starts from a
     point in the gradient path where $\pi_{d_e}(\bma)$ is the
     maximal label, and we go through the same moves as Row~$i$, with $i=e$, but
     we skip the last two steps: we do not add
     $\pi_{\{d_e,d_e\}}(\bma)$ (since we already have it) and we do
     not delete $\pi_{d_e}(\bma)$.  Now the very last set of indices
     in the table are those of $\pi_L(\bma)$ appearing in $\sigma''$,
     and hence we have built a gradient path between $\sigma'$ and
     $\sigma''$.
     
     \smallskip
     
  Finally, if $k=d_s$ or, in other words, $e=s$ we build a more
     compact version of the path above. The path below with initial
     point  $\sigma'$ and  end point  $\sigma''$ shows only the
     indices $L$ appearing in each $\sigma_i$ along the way.

{\tiny
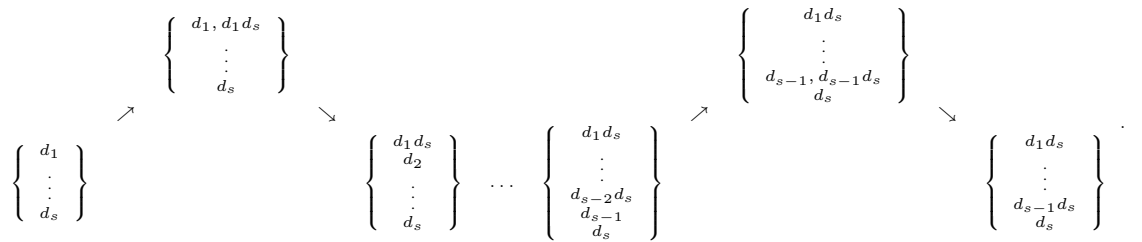
 $$\begin{array}{ccccccccccc}
    && {\left \{ \begin{array}{c}
     d_1, d_1d_s \\ \vdots \\d_s 
         \end{array} \right \}}
    &&&&&&
    {\left \{ \begin{array}{c}
     d_1d_s\\\vdots\\ d_{s-1} ,d_{s-1}d_s \\ d_s 
      \end{array} \right \}}
    &&  \\
    &\nearrow &&\searrow &&&&\nearrow &&\searrow & \\ 
    {\left \{ \begin{array}{c}
     d_1\\\vdots\\d_s 
      \end{array} \right \}}
    &&&&
    {\left \{ \begin{array}{c}
     d_1d_s\\ d_2\\\vdots\\d_s 
      \end{array} \right \}}
    &\ldots&
    {\left \{ \begin{array}{c}
     d_1d_s\\\vdots\\d_{s-2}d_s\\d_{s-1}\\d_s 
      \end{array} \right \}}
    &&&&
    {\left \{ \begin{array}{c}
     d_1d_s\\\vdots\\d_{s-1}d_s\\d_s 
      \end{array} \right \}}   
     \end{array}.$$
     
}
 \end{proof}

 We next show that
  \cref{l:destination} is reversible, that is, we prove that all
  gradient paths connect cells of the given form. In order to do so,
  we first show in \cref{l:critical-k} that all critical cells
  contained in $\osigma(\bma, D)$ of dimension one less either contain
  $\bma$ or have the form given in~\cref{l:destination}, and then we
  establish in \cref{l:still} a few basic facts about the cells and
  their orders.

\begin{lemma}\label{l:critical-k} Let $\bma \in \cM_r$ and
  $\varnothing \neq D \subseteq \supp(\bma) \ssm \{1\}$.  If $\sigma'$
  is a $(|D|-1)$-dimensional $A$-critical cell contained in
  $\osigma(\bma,D)$ and $\bma\notin \sigma'$, then $\sigma'
  =\sigma(\pi_k(\bma), \, D\ssm \{k\})$ for some $k\in D$.
\end{lemma}

\begin{proof}
Suppose that $|D| = 1$.  Then $\osigma(\bma, D) = \sigma(\bma, D) = \{\bma,
\pi_k(\bma)\}$ where  $D=\{k\}$.  If $\sigma'$ is a 0-critical cell in
this set and $\bma  \notin \sigma'$, then  $\sigma' = \{\pi_k(\bma)\}$,
hence $\sigma' =\sigma(\pi_k(\bma), \varnothing)$.

Now assume that $|D| \geq 2$, and let $\sigma'
  =\sigma(\bmb, D')$ for some $\bmb \in \cM_r$ and $D'\subseteq
  \supp(\bmb)\ssm \{1\}$ with $|D'|=|D|-1$ and with $\bma \not\in
  \sigma'$. Then $$\bmb \in \sigma(\bmb, D') = \sigma' \subseteq
  \osigma(\bma, D).$$ Since $\bmb\in \osigma(\bma,D)$ and $\bma \ne
  \bmb$, for some nonempty $L\subseteq D$ we
  have $$\bmb=\pi_L(\bma).$$

\begin{claim} $D'\subseteq D\ssm L$.
\end{claim}

\noindent{\it Proof of Claim.}  Let $d\in D'$.
Then $$\pi_d(\pi_L(\bma)) = \pi_d(\bmb) \in \sigma(\bmb, D') \subseteq
\osigma(\bma,D).$$ Consequently, for some $L'\subseteq D$
\begin{equation} \label{e:pre-m}
\pi_d(\pi_L(\bma))=\pi_{L'}(\bma).
\end{equation}

By \cref{n:critical-cells,d:generators}, this implies that 
$$\sum_{i\in L} \mathbf{e}_{i} 
-\sum_{i\in L} \mathbf{e}_{\tau(i)} + \mathbf{e}_d -\mathbf{e}_{\tau(d)} =
\sum_{i\in L'} \mathbf{e}_{i}
- \sum_{i\in L'} \mathbf{e}_{\tau(i)}.$$

All indices in $L \cap L'$ appear on both sides equally, so we can
cancel those out and assume that $L \cap L'=\varnothing$. Now let $j =
\max\{L\cup L' \cup \{d\}\}$. If $j\neq d$, then considering the fact
that $\tau(i)<i$ for each $i\neq 1$, $j \in L \cap L'$ (since
$\mathbf{e}_j$ must appear on both sides of the equation). This is
impossible as we assumed $L\cap L' =\varnothing$. Therefore $j=d$. It
follows that $d \in L' \subseteq D$, and since
$L\cap L' = \varnothing$, it also follows that $d\notin L$. 
Therefore $D' \subseteq D \ssm L$, which proves our claim.

Since $\sigma'$ is a $(|D|-1)$-cell, $\sigma'$ has exactly $|D|$
elements; on the other hand, $\sigma' =\sigma(\bmb, D')$ has $|D'|+1$
elements.  Therefore, $|D'|=|D|-1$.  Under the assumption that
$D'\subseteq D\ssm L$, it follows that $L$ must have just
one element, which gives $$\bmb=\pi_k(\bma) \qand
D'=D\ssm \{k\}$$ for some $k\in D$, which ends our argument.
\end{proof}

\begin{lemma} 
\label{l:still}
Let $\ba, \bb,\bc \in \cN_r$, such that $\bmb,\bmc \in \osigma(\bma,D)$,
where $D \subseteq \supp(\ba) \ssm \{1\}$.
Then there exist $L, L' \subset D$ such that
  $$\bmb=\pi_{L'}(\bma), \quad \bmc=\pi_{L}(\bma).$$ Set $ k=\max \left(
(L\cup L')\ssm(L\cap L') \right).$ Then
$$\bmc \prec \bmb \iff  k \in L$$ 
and in this case  $\dom(\bmb,\bmc)=k.$
In particular, $\pi_k(\bmb)\in \osigma(\bma,D)$.
\end{lemma}

\begin{proof}  The existence of $L, L'$ follows directly from the definition of $\osigma(\bma, D)$ in \cref{n:critical-cells}.
By assumption we know $\bb=\pi_{L'}(\ba)$
    and $\bc=\pi_{L}(\ba)$.  Let $\bb=(b_1,\ldots,b_q)$ and
    $\bc=(c_1,\ldots,c_q)$.  By \cref{d:generators}, $\bc \prec \bb$ if
    and only if $b_j - c_j> 0$ where $j$ is the largest index with
    $b_j \neq c_j$. Moreover
    
  \begin{align*}
  \bb- \bc & = \pi_{L'}(\ba) - \pi_{L}(\ba)\\
  &= \left (\ba + \sum_{i\in L'} \mathbf{e}_{\tau(i)} - \sum_{i\in L'} \mathbf{e}_{i} \right ) -
  \left ( \ba + \sum_{i\in L} \mathbf{e}_{\tau(i)} - \sum_{i\in L} \mathbf{e}_{i} \right) \\
  &=  \sum_{i\in L' \ssm L\cap L'} \mathbf{e}_{\tau(i)} + \sum_{i\in L \ssm L\cap L'} \mathbf{e}_{i} - \left ( \sum_{i\in L' \ssm L\cap L'} \mathbf{e}_{i} +\sum_{i\in L \ssm L\cap L'} \mathbf{e}_{\tau(i)} \right ). 
\end{align*}

Since $\tau(i) < i$ when $i\neq 1$, if  $k \in L$, then it follows immediately
that $\bc \prec \bb$ and $\dom(\bb,\bc)=k$. If $k\in L'$ a similar equation for
$\bc- \bb$ shows that $\bb \prec \bc$ and $\dom(\bc,\bb)=k$.
  
 \end{proof}

 We are now ready to prove a converse statement of \cref{l:destination}, which
 is the final ingredient needed to prove \cref{t:critical-cells}.

\begin{lemma}\label{l:grad} Let $\bma \in \cM_r$,
  $D \subseteq \supp(\bma) \ssm \{1\}$ and $\sigma=\sigma(\bma,D)$ an
  $A$-critical cell of $X$ of dimension $|D|$, and $\sigma'= \sigma
  \ssm \{\bma\}$ a noncritical cell of $\sigma$.  Suppose $\sigma''$
  is an $A$-critical cell of $X$ of dimension $|D|-1$ which is
  connected to $\sigma'$ via the gradient path $$\mathcal P 
  \colon \sigma'=\sigma_0\to \sigma_1 \to \cdots \to \sigma_{h-1} \to
  \sigma_h=\sigma''$$ in $G^A_X$.  Then for some $k\in
  D$ $$  \sigma''=\sigma(\pi_k(\bma), D\ssm \{k\}).$$
\end{lemma}

\begin{proof} 
  We will show that for each $i$, if
  $\sigma_i\subseteq \osigma(\bma,D)$, then $\sigma_{i+1}\in
  \osigma(\bma,D)$. The arrow $$\sigma_{i} \to \sigma_{i+1}$$ in
  $\mathcal P$ could be one of two forms. If it is a downward arrow, then by \cref{d:grad} we have an 
  inclusion $\sigma_{i+1} \subseteq \sigma_i \subseteq
  \osigma(\bma,D)$.

  If the arrow is pointing up, from \cref{d:grad} we
  know $$\sigma_{i+1}= \sigma_i \cup \{\pi_k(\bmb)\},$$ where
  $\bmb=\max(\sigma_i)$ and $k=\dom(\sigma_{i+1}) \ne -\infty$.  We have thus $k=\dom(\bmb, \bmc)$ for some $\bmc\in \sigma_{i+1}\ssm\Pi(\bmb)$ with $\bmc\prec \bmb$. In particular, we have $\bmc\in \sigma_i$. Since we assumed $\sigma_i\subseteq \osigma(\bma,D)$, we know $\bmb, \bmc\in \osigma(\bma,D)$.     By \cref{l:still},
  $\pi_k(\bmb) \in \osigma(\bma,D)$ and so
$$\sigma_{i+1}\subseteq \osigma(\bma,D).$$

Since $\sigma_0=\sigma' \subseteq \osigma(\bma,D)$, the above
argument, applied inductively over $i$ implies that $$\sigma'' \subseteq
\osigma(\bma,D).$$

Note that $\bma \notin \sigma''$, since $\max(\sigma_i)$ is
non-increasing along $\mathcal P$ (see \cref{d:grad}),  $\bma \notin \sigma'$, and $\bma$ is larger than all
elements of $\sigma'$, so it cannot be added along the gradient
path. By \cref{l:critical-k} for some $k\in
D$, $$\sigma''=\sigma(\pi_k(\bma), D\ssm \{k\}).$$\end{proof}

Now that the gradient paths and the forms of the associated critical
cells have been established, the proof of \cref{t:critical-cells}
follows.

\begin{proof}[Proof of \cref{t:critical-cells}]  Part~(1) of the
  statement follows directly from \cref{mainresult}(3) and the
  description of the $A$-critical cells as $\sigma(\bma,D)$.  For part~(2)
  suppose $c=\sigma_A$ and $c'=\sigma'_A$ are $i$ and $(i-1)$ cells of
  $X_A$, respectively.  If $c' \leq c$, then by \eqref{e:cells} there
  are two possibilities. The first one is $\sigma' \subseteq \sigma$,
  in which case, since $\sigma'$ is $A$-critical, for some $k\in D$,
  $\sigma'=\sigma(\bma, D\ssm \{k\})$. Otherwise there is a gradient
  path from an $(i-1)$-cell  $\sigma''$ of $\sigma$ to $\sigma'$. Then
  by \cref{d:grad} $\sigma''$ must be a non-critical cell of
  $\sigma$. So $\sigma''=\sigma\ssm \{\bma\}$ and by \cref{l:grad} we
  must have $\sigma'=\sigma(\pi_k(\bma), D \ssm \{k\})$ for some
  $k \in D$.

  Conversely, if $\sigma'=\sigma(\bma, D\ssm \{k\})$ for some $k \in
  D$, then by \eqref{e:cells} $c' \leq c$. Suppose
  $\sigma'=\sigma(\pi_k(\bma), D \ssm \{k\})$.  If $|D|=1$, we have
  $D=\{k\}, \sigma'=\sigma(\pi_k(\bma), \varnothing)$ and
  $\sigma=\sigma(\bma, \{k\})= \{\bma, \pi_k(\bma)\}$. In this case,
  $\sigma' \subseteq \sigma$ and $c' \leq c$ by \eqref{e:cells}. If
  $|D| \geq 2$, then by \cref{l:destination} there is a gradient path
  from $\sigma \ssm \{\bma\}$ to $\sigma'$, which again by
  \eqref{e:cells} implies that $c' \leq c$. This concludes the proof.

\end{proof}

\begin{example} \label{e:list} With $I$ as in 
  \cref{shadedtriangle}, the critical cells for any $\bma \in \cM_2$
  are depicted in \cref{f:critical}. Note that in all
  cases $|\supp(\bma) \ssm \{1\}| \leq 2$, so we can say exactly what
  the cell order in the Morse complex look like from
  \cref{t:critical-cells}.  For example in the case of
  $\bm^{(0,1,1)}$, we note that $\supp(\bm^{(0,1,1)})=\{2,3\}$, so
  we let $D=\{2,3\}$, then $\sigma=\sigma(\bm^{(0,1,1)}, D)$ is a
  critical $2$-cell of $X$ corresponding to a $2$-cell $\sigma_A$ of
  the Morse complex $X_A$. The $1$-cells of $X_A$ contained in
  $\sigma_A$ will be $\sigma'_A$ where $\sigma'$ ranges over the
  following critical $1$-cells of $X$:
$$\sigma(\bm^{(0,1,1)}, \{2\}),
\quad \sigma(\bm^{(0,1,1)}, \{3\}),
\quad \sigma(\pi_2(\bm^{(0,1,1)}), \{3\}),
\quad \sigma(\pi_3(\bm^{(0,1,1)}), \{2\}).
$$

We list the cell order in the Morse complex in the
table below, using critical cells.

$$\begin{array}{c|c|c}
\bma & \bpi(\bma) \mbox{ corresponding to an } 
& \mbox{subsets of  $\bpi(\bma)$ corresponding to}\\
& i-\mbox{cell of } X_A 
& (i-1)-\mbox{sub-cells in $X_A$}\\
\hline
&&\\
\bm^{(0,1,1)} & \{\bm^{(0,1,1)},\bm^{(1,0,1)},\bm^{(0,2,0)}\} &
\{\bm^{(0,1,1)},\bm^{(1,0,1)}\} \\
&&\{\bm^{(0,1,1)},\bm^{(0,2,0)}\}\\
&&\{\bm^{(1,0,1)},\bm^{(1,1,0)}\} \\ 
&&\{\bm^{(0,2,0)},\bm^{(1,1,0)}\} \\ 
\bm^{(1,1,0)} & \{ \bm^{(1,1,0)}, \bm^{(2,0,0)} \} &
\{\bm^{(1,1,0)}\}, \{\bm^{(2,0,0)}\}\\ 
\bm^{(1,0,1)} & \{ \bm^{(1,0,1)}, \bm^{(1,1,0)} \} &
\{\bm^{(1,0,1)}\},\{\bm^{(1,1,0)}\}\\
\bm^{(0,2,0)} & \{ \bm^{(0,2,0)}, \bm^{(1,1,0)}  \} &
\{\bm^{(0,2,0)}\},\{\bm^{(1,1,0)}\}\\  
\bm^{(0,0,2)} & \{ \bm^{(0,0,2)}, \bm^{(0,1,1)}  \} &
\{\bm^{(0,0,2)}\}, \{\bm^{(0,1,1)}\} \\
\bm^{(2,0,0)} & \{ \bm^{(2,0,0)}  \} & \varnothing\\ 
&&\\
\end{array}
$$
 One can show that the $2$-cell of the Morse complex can be represented by  the whole square in \cref{f:critical} of  \cref{shadedtriangle}.
\end{example}

  \begin{example}\label{e:new} Consider the ideal
    $I=(xyzv, xyw, yuvw, xuvw)$ in the polynomial ring
    $S=\kk[x,y,z,u,v,w]$, and let $\Delta$ be the simplicial complex
    $\F(I)^c$, pictured on the left in \cref{f:new}.  Then $\Delta$ is
    a quasi-tree with facet ordering $F_1,F_2,F_3,F_4$ (this is one of
    several possible orders). By \cref{t:FH} $\pd(I)=1$. Using the
    above facet order we set $\tau(2)=1$ and $\tau(3)=\tau(4)= 2$. We
    label the monomial generators of $I$ accordingly as $$m_1=xyzv,
    m_2=xyw, m_3=yuvw, m_4=xuvw.$$
    
\begin{figure}
\begin{tikzpicture}
\coordinate (A) at (0, 0);
\coordinate (B) at (1.5, 0);
\coordinate (C) at (3, -0.65);
\coordinate (D) at (3,0.65 );
\coordinate (E) at (4.5, -0.15);
\coordinate (F) at (4.5, -1.25);
 \draw[black, fill=black] (A) circle(0.06);
 \draw[black, fill=black] (B) circle(0.04);
 \draw[black, fill=black] (C) circle(0.04);
 \draw[black, fill=black] (D) circle(0.04);
 \draw[black, fill=black] (E) circle(0.04);
 \draw[black, fill=black] (F) circle(0.04);
 \draw[line width=0.2mm, fill=lightgray ] (D) -- (B) -- (C)  -- cycle node[ pos=0.5, xshift=-4mm] {$F_2$};
 \draw [line width=0.2mm]  (A) -- (B)  node[pos=0.5, above] {$F_1$};
\draw[line width=0.2mm] (E) -- (C) node[pos=0.5, above] {$F_3$};;
\draw[line width=0.2mm] (F) -- (C) node[pos=0.5, below] {$F_4$};;
\node[label = below :$w$] at (A) {};
\node[label = below  :$u$] at (B) {};
\node[label = below left :$z$] at (C) {};
\node[label = above :$v$] at (D) {};
\node[label = above :$x$] at (E) {};
\node[label = below :$y$] at (F) {};
\end{tikzpicture}
\begin{tikzpicture}
\coordinate (A) at (0, 0);
\coordinate (B) at (1, 1);
\coordinate (C) at (-1, 1);
\coordinate (D) at (0,2 );
\coordinate (E) at (2, 1);
\coordinate (F) at (-2, 1);
\coordinate (G) at (0, -2);
\coordinate (H) at (-1, -1);
\coordinate (I) at (1, -1);
\coordinate (J) at (0, -3);
 \draw[ draw = none,fill=lightgray ] (B) -- (I) -- (G)  -- cycle;
  \draw[draw = none, fill=lightgray] (A) -- (H) -- (G)  -- cycle;
  \draw[draw = none, fill=lightgray] (A) -- (D) -- (B)  -- cycle;
 \draw[black, fill=black] (A) circle(0.06);
 \draw[black, fill=black] (B) circle(0.04);
 \draw[black, fill=black] (C) circle(0.04);
 \draw[black, fill=black] (D) circle(0.04);
 \draw[black, fill=black] (E) circle(0.04);
 \draw[black, fill=black] (F) circle(0.04);
  \draw[black, fill=black] (G) circle(0.04);
 \draw[black, fill=black] (H) circle(0.04);
 \draw[black, fill=black] (I) circle(0.04);
  \draw[black, fill=black] (J) circle(0.04);
\draw[line width=0.2mm] (A) -- (B);
\draw[line width=0.2mm] (A) -- (C);
\draw[line width=0.2mm] (B) -- (D);
\draw[line width=0.2mm] (D) -- (C);
\draw[line width=0.2mm] (C) -- (F);
\draw[line width=0.2mm] (B) -- (E);
\draw[line width=0.2mm] (A) -- (G);
\draw[line width=0.2mm] (H) -- (C);
\draw[line width=0.2mm] (H) -- (G);
\draw[line width=0.2mm] (B) -- (I);
\draw[line width=0.2mm] (I) -- (G);
\draw[line width=0.2mm] (G) -- (J);
\node[label = above :$\bm^{(0,2,0,0)}$] at (A) {};
\node[label = below right :$\bm^{(0,1,1,0)}$] at (B) {};
\node[label = below left :$\bm^{(1,1,0,0)}$] at (C) {};
\node[label = above :$\bm^{(1,0,1,0)}$] at (D) {};
\node[label = above right :$\bm^{(0,0,2,0)}$] at (E) {};
\node[label = above left :$\bm^{(2,0,0,0)}$] at (F) {};
\node[label = left :$\bm^{(0,1,0,1)}$] at (G) {};
\node[label = above left :$\bm^{(1,0,0,1)}$] at (H) {};
\node[label = above  right:$\bm^{(0,0,1,1)}$] at (I) {};
\node[label = below :$\bm^{(0,0,0,2)}$] at (J) {};

\end{tikzpicture}\caption{The figures for \cref{e:new}}\label{f:new}

\end{figure}

    By \cref{mainresult}, $I^2$ has a free resolution supported on a
    CW complex whose cells are in one-to-one correspondence with the
    $A$-critical cells of the Taylor complex of $I^2$. We record these
    critical cells in the right picture in \cref{f:new}.  Note that
    the faces $$\{\bm^{(1,0,1,0)}, \bm^{(0,2,0,0)} \},
    \{\bm^{(1,0,0,1)}, \bm^{(0,2,0,0)} \}, \{\bm^{(0,1,0,1)},
    \bm^{(0,1,1,0)} \}$$ are not $A$-critical cells, hence they are
    not edges in the diagram.

\cref{t:critical-cells} allows us to exactly determine to which
$2$-cell of the Morse complex each of the shaded triangles which are
missing an edge corresponds.  We focus on one such triangle coming
from the critical cell for the monomial generator
$$\bm^{(0,0,1,1)}=m_3m_4=xyu^2v^2w^2$$ of $I^2$. The $2$-dimensional
critical cell of the Taylor complex of $I^2$
$$\sigma=\bpi(\bm^{(0,0,1,1)})=\sigma(\bm^{(0,0,1,1)},\{3,4\})=\{\bm^{(0,0,1,1)},\bm^{(0,1,0,1)},\bm^{(0,1,1,0)}\}$$
corresponds to a $2$-cell $\sigma_A$ of the Morse complex which, by
\cref{t:critical-cells}, contains the
$1$-cells $$(\sigma_1)_A,(\sigma_2)_A,(\sigma_3)_A,(\sigma_4)_A$$ where
$$\begin{array}{ll}
     \sigma_1=\{\bm^{(0,0,1,1)},\bm^{(0,1,0,1)}\},
  &  \sigma_2=\{\bm^{(0,0,1,1)},\bm^{(0,1,1,0)}\},\\  
     \sigma_3=\{\bm^{(0,1,0,1)},\bm^{(0,2,0,0)}\},
  &  \sigma_4=\{\bm^{(0,1,1,0)},\bm^{(0,2,0,0)}\}.\\
\end{array}
$$ The $2$-cell $\sigma_A$ can therefore be visualized as the following
shaded square.
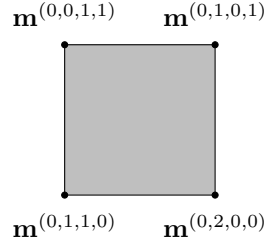
\begin{center}
\begin{tikzpicture}
\coordinate (B) at (2, 0);
\coordinate (C) at (4, 0);
\coordinate (D) at (4,2 );
\coordinate (E) at (2, 2);
 \draw[draw = none, fill=lightgray] (B) -- (C) -- (D)  -- (E) ;
 \draw[black, fill=black] (B) circle(0.04);
 \draw[black, fill=black] (C) circle(0.04);
 \draw[black, fill=black] (D) circle(0.04);
 \draw[black, fill=black] (E) circle(0.04);
\draw[-] (B) -- (C);
\draw[-] (B) -- (E);
\draw[-] (C) -- (D);
\draw[-] (E) -- (D);
\node[label = below :$\bm^{(0,1,1,0)}$] at (B) {};
\node[label = below :$\bm^{(0,2,0,0)}$] at (C) {};
\node[label = above :$\bm^{(0,1,0,1)}$] at (D) {};
\node[label = above :$\bm^{(0,0,1,1)}$] at (E) {};
\end{tikzpicture}
\end{center}

\end{example}


\section{Minimality of the Morse resolution} \label{minimalsection}


 This section is devoted to establishing the minimality of the free
 resolution supported by the Morse complex $X_A$ described in
 \cref{mainresult}. What we need to establish (\cite[Lemma~7.5]{BW})
 is that if $c' \leq c$ are two cells of the CW complex $X_A$ with
 $\dim(c') = \dim(c) - 1$, then $\lcm(c) \neq \lcm(c')$.

   In our case, the critical cells are of the form $\sigma(\bma,D)$
   where $\bma \in \cM_r$ and $D\subseteq \supp(\ba)\ssm\{1\}$ 
   using \cref{n:critical-cells}.  By applying
   \cref{t:critical-cells} we proceed to prove that no two embedded
   cells of $X_A$ come from $A$-critical cells of $X$ with the
   same lcm label.

 Our first statement describes the monomial label of
     the face $\sigma(\bma,D)$ of the Taylor complex. To find the label, multiply $\bma$
     by the free vertices of the leaf $F_j$ of
     the quasi-tree $\langle F_1,\ldots,F_j \rangle$ in
     \cref{construct} for all $j \in D$.

\begin{proposition}
\label{p:sigmalcm}
Let $\bma\in \cM_r$, $D\subseteq \supp(\ba)\ssm\{1\}$.  Then
$$\lcm \sigma(\bma, D) = \bma \displaystyle \prod_{j \in D} \prod_{x \in
  F_{j}\ssm F_{\tau(j)}} x.$$ 
\end{proposition}

\begin{proof} We use induction on $d=|D|$. If $d=0$, then $D=\varnothing$ and
  by \cref{n:critical-cells} we have $\sigma(\bma, D) = \bma$, so
  there is nothing to prove.
  If $d=1$, then $D=\{k\}$ and (see \cref{construct})
  \begin{align*}
    \lcm \sigma(\bma,\{k\})
    = &\lcm \left ( \bma, \pi_k(\bma)\right )
    =  \lcm \left ( \bma, \frac{\bma \cdot m_{\tau(k)}}{m_k} \right )\\
    = &  \bma \cdot \displaystyle  \prod_{x \mid m_{\tau(k)}, \ x \nmid m_k} x 
    =  \bma \cdot \displaystyle \prod_{x \notin F_{\tau(k)}, \ x \in F_k} x
    = \bma \cdot \displaystyle  \prod_{x \in F_{k}\ssm F_{\tau(k)}} x. \\
  \end{align*}

  If $d>1$, and $k=\max(D)$, let $D'=D \ssm \{k\}$. Then, once again
  using \cref{construct} as in the base case, plus the induction
  hypothesis on $|D'|$, we have
  \begin{align*} \lcm \sigma(\bma,D)
    =&\lcm \left( \lcm \sigma(\bma,D'), \pi_k(\bma) \right ) \\
    =&\lcm \left( \bma \displaystyle \prod_{j \in D'} \prod_{x \in
      F_{j}\ssm F_{\tau(j)}} x, \ \frac{\bma \cdot m_{\tau(k)}}{m_k} \right )\\
    =& \bma \displaystyle \prod_{j \in D'} \prod_{x \in
      F_{j}\ssm F_{\tau(j)}} x \cdot \prod_{x \in F_{k}\ssm F_{\tau(k)}} x \\
        =& \bma \displaystyle \prod_{j \in D} \prod_{x \in
      F_{j}\ssm F_{\tau(j)}} x. \\
  \end{align*}
\end{proof}

\begin{theorem}[{\bf The resolution supported on $X_A$ is
      minimal}]\label{t:minimality} The free resolution that is
  supported on the CW Morse complex $X_A$, with $A$ being the acyclic
  matching as in \eqref{A}, is minimal.
\end{theorem}

\begin{proof} By ~\cite[Lemma~7.5]{BW} the resolution is
minimal if and only if for $\sigma_A$ an $i$-cell and $\sigma_A'$ an
$(i-1)$-cell of $X_A$ such that $\sigma_A'\le \sigma_A$, then
$\lcm(\sigma)\ne \lcm(\sigma')$.  

 Using \cref{t:critical-cells} we know that, if $\sigma=\sigma(\bma,
 D)$ with $\bma\in \cM_r$ and $D \subseteq \supp(\bma \ssm \{1\})$,
 then $\sigma'=\sigma(\bma,D \ssm\{k\})$ or
 $\sigma'=\sigma(\pi_k(\bma),D \ssm\{k\})$ for some $k\in D$.

 \begin{itemize}

 \item If $\sigma'=\sigma(\bma,D \ssm\{k\})$, then by
   \cref{p:sigmalcm}
   $$\lcm \sigma = \lcm \sigma' \cdot \prod_{x \in F_{k}\ssm
     F_{\tau(k)}} x.$$ Since $F_k \ssm F_{\tau(k)}$ consists of the
   free vertices of the leaf $F_k$ of the quasi-tree
   $\langle F_1,\ldots,F_k \rangle$, we have 
   $F_k \ssm F_{\tau(k)} \neq \varnothing$. Thus 
   $\lcm(\sigma) \neq \lcm(\sigma')$.

 \item If $\sigma'=\sigma(\pi_k(\bma),D \ssm\{k\})$, then suppose
   $\lcm(\sigma)=\lcm(\sigma')$. From \cref{p:sigmalcm} it  follows
   that
   $$\pi_k(\bma) \displaystyle \prod_{j \in D\ssm \{ k \}} \prod_{x \in
     F_{j}\ssm F_{\tau(j)}} x
   = \bma \displaystyle \prod_{j \in D} \prod_{x \in
     F_{j}\ssm F_{\tau(j)}} x.$$ Replacing $\pi_k(\bma)$ in the equation  and simplifying both sides, we have
   $$ \displaystyle \frac{\bma \cdot m_{\tau(k)}}{m_k}
   = \bma \displaystyle \prod_{x \in
     F_{k}\ssm F_{\tau(k)}} x,$$ which with further simplification results in
   $$m_{\tau(k)} =  m_k \prod_{x \in F_{k}\ssm
     F_{\tau(k)}} x.$$ This last equality implies that
   $m_k \mid m_{\tau(k)}$ which (see \cref{construct}) means that
   $F_{\tau(k)} \subseteq F_k$, a contradiction to $F_{\tau(k)}$ and
       $F_k$ both being facets of a simplicial complex, and $\tau(k) \lneq k$.
 \end{itemize}\end{proof}

Note that an alternative approach to
  showing minimality is to explicitly show that the Morse complex is
  isomorphic to the convex cellular complex defined in \cite{koszul}.

  \begin{corollary}[{\bf The projective dimension of $I^r$}]\label{c:pd-I^r}
    If $I$ is  generated by $q$ square-free monomials in a
   polynomial ring $S$, $I$ has projective dimension one , and $r$ is
    a positive integer, then
\begin{equation}
  \pd_S(I^r)=\begin{cases} q-1 & \qif r\ge q-1\\
                          r & \qif r< q-1
\end{cases}
\end{equation}
\end{corollary}

  \begin{proof} 
  By \cref{mainresult,t:minimality} we know that the Morse complex
  $X_A$ from \cref{mainresult} supports a minimal free resolution of
  $I^r$ and for any $k$, the $k$-cells of $X_A$ are in one-to-one
  correspondence with the $A$-critical $k$-cells of $X=\taylor(I)$.

Therefore the value of $\pd_S (I^r)$ is equal to the largest size of a
critical cell.  Recall that the critical cells have the form
$\sigma(\bma, D)$ with $\bma\in \cM_r$ and $D\subseteq
\supp(\ba)\ssm\{1\}$, and thus the largest cell has
$D=\supp(\ba)\ssm\{1\}$. We have then
$$\pd_S(I^r)=\max\{|\supp(\ba)\smallsetminus \{1\}| \mid \ba\in
(\NN \cup \{0\})^q, a_1+a_2+\dots+a_q=r\}.
$$

The maximum in the expression above is obtained by choosing $\ba$
such that $a_i=1$ for as many values of $i>1$ as possible. When $q\le
r+1$, the maximum is obtained, for example, when $\ba=(0,1,\dots, 1,
r-q+2)$; the maximum is $q-1$ in this case.  When $q> r+1$, the
maximum is achieved, for example, when $\ba=(0,1,\dots, 1, 0,\dots, 0)$, where
precisely $r$ entries are non-zero; the maximum is thus $r$ in this
case.
\end{proof}

It is worth noting that the precise formula for the projective
dimension of all powers also provides a precise formula for the depths
and in particular, the formula above pinpoints exactly where this
sequence stabilizes. This stabilizing point is referred to as the {bf
  index of depth stability} in the literature and is denoted
$\dstab(I)$. Finding bounds for $\dstab(I)$ is an active area.

  \begin{corollary}\label{c:dstab}
If $I$ is generated by $q$ square-free monomials and has projective dimension $1$, then $\dstab (I)=n-q+1$. \end{corollary}
 
\begin{proof}
This follows immediately from \cref{c:pd-I^r} and the Auslander-Buchsbaum formula (see \cite[Theorem 19.9]{eisenbud}).
\end{proof}

Other invariants that can be read from an explicit minimal graded free
resolution of a graded module are maximal shifts and Castelnuovo-Mumford
regularity. In \cref{ss:supported} we defined the multi-graded
Betti numbers of a monomial ideal $I$, denoted $\beta_{i,m}(I)$ for
each integer $i$ and each $m\in \LCM(I)$. For each integer $i$, the
invariant
\begin{equation}\label{e:shift}
t_i^S(I)=\sup\{j\in \mathbb Z\mid \beta_{i,m}(I)\ne 0 \text{ for some $m\in \LCM(I)$ with $\deg(m)=j$}\}
\end{equation}
is precisely the {\bf maximal shift} (with respect to total degree) of
a free module that occurs in the $i$th component of a minimal graded
free resolution of $I$.  The {\bf Castelnuovo-Mumford regularity} of
an ideal $I$ is defined by
\begin{equation}\label{e:regularity}
\reg_S(I)=\sup_{i\ge 0}\{t_i^S(I)-i\}\,.
\end{equation}
We establish a formula for maximal shifts and
regularity in the case of interest for our paper.

\begin{corollary}
\label{c:shifts}
If $I$ has projective dimension $1$ and is minimally generated by
square-free monomials $m_1,\ldots,m_q$ in a polynomial ring $S$ over a
field, then for all $i\le r$ we have:
\begin{align*}
  t_i^S(I^r)= &\max \Big \{\sum_{j\in D} \deg(\lcm(m_j, m_{\tau(j)})) \mid D\subseteq [q]\smallsetminus \{1\}, |D|=i\Big \}\\
  &+ \quad (r-i)\max\{ \deg(m_j) \mid j\in [q]\}  \,.
\end{align*}
In particular, $r\mapsto t_i^S(I^r)$ is a linear function when $r\ge i$. 
\end{corollary}

\begin{proof}
The description  of the minimal free resolution of $I^r$ as the free resolution supported on the CW complex $X_A$  (\cref{mainresult,t:minimality}) shows 
\begin{align*}
\beta_{i,m}(I^r)\ne 0\iff & m=\lcm\sigma(\bma,D) \text{ for some   $\bma\in \cM_r$}\\ &\text{and $D\subseteq \supp(\ba)\smallsetminus\{1\}$ with $|D|=i$}
\end{align*}
and hence, using \cref{p:sigmalcm}, we have  
\begin{align*}
  t_i^S(I^r)&=\max \{\deg(\lcm\sigma(\bma, D))\mid \bma\in \cM_r, D\subseteq \supp(\ba)\smallsetminus\{1\}, |D|=i \}\\
  &=\max \{\deg \Big (\bma \displaystyle \prod_{j \in D} \prod_{x \in
  F_{j}\ssm F_{\tau(j)}} x  \Big ) \mid \bma \in \cM_r, D \subseteq \supp(\ba)\smallsetminus\{1\}, |D|=i \}.\\
\end{align*}

Let $u\in [q]$ be such that $$\deg(m_u)=\alpha=\max \{\deg(m_j) \mid j\in [q] \}\,.$$ 
Now  let $\ba\in
\cM_r$ and $D=\{k_1, \dots, k_i\}\subseteq
\supp(\ba)\smallsetminus\{1\}$. If  ${\bf e}_1, \dots, {\bf e}_q$ is the standard basis of $\mathbb R^q$,  set
$$\bb=\ba-\sum_{j\in D}{\bf e}_j\in \cM_{r-i} \quad \text{ so that }\quad 
\bma=m_{k_1}\dots m_{k_i}\bmb.$$ 
Using \cref{p:sigmalcm}, we see that 
\begin{align*}
\deg (\lcm \sigma(\bma, D))&=\deg(\bma)+\sum_{j\in D}|F_j\smallsetminus F_{\tau(j)}|\\
&=\deg(\bmb)+\sum_{j\in D}\left(\deg(m_j)+|F_j\smallsetminus F_{\tau(j)}|\right)\\
&\le \alpha(r-i)+\sum_{j\in D}\left(\deg(m_j)+|F_j\smallsetminus F_{\tau(j)}|\right)\,.\\
\end{align*}

On the other hand, if we let $$\bmc=m_{k_1}\dots m_{k_i}(m_u)^{r-i}
\in \cM_r$$ we observe that $D \subseteq \supp(\bmc)$ and, by
\cref{p:sigmalcm}, $$\alpha(r-i)+\sum_{j\in
  D}\left(\deg(m_j)+|F_j\smallsetminus F_{\tau(j)}|\right)=
\deg(\lcm\sigma(\bmc, D)).$$ This implies that
$$
t_i^S(I^r)=\alpha(r-i)+\max \Big \{\sum_{j\in D}\left(\deg(m_j)+|F_j\smallsetminus F_{\tau(j)}|\right)\colon D\subseteq [q]\smallsetminus \{1\}, |D|=i \Big \}.
$$ The formula in the statement follows  by noting that, by
\cref{p:sigmalcm} for all $j\in[q]$
$$
\deg(\lcm(m_j, m_{\tau(j)}))=\deg(m_j)+|F_j\smallsetminus F_{\tau(j)}|.
$$ 
\end{proof}

It is known that the regularities of the powers of a homeogeneous
ideal in a polynomial ring are asymptotically linear
(\cite[Theorem~4.7]{powerideals}, \cite{CDHT,K00,TW}). In other words,
there are integers $a,b$ such that $$\reg(I^r)=ar+b$$ for all
$r\gg 0$. \cref{c:reg} below fine tunes this fact for square-free
monomial ideals of projective dimension~1.  

\begin{corollary}[{\bf The regularity of $I^r$}]\label{c:reg}
  If $I$ has projective dimension $1$ and is minimally generated by
square-free monomials $m_1,\ldots,m_q$ in a polynomial ring $S$ over a
field, then
$$\reg_S(I^r)=\alpha r + (1-q)\alpha +\reg_S(I^{q-1})$$
for all $r\ge q-1$ where $\alpha=\max \{\deg(m_j) \mid j\in [q] \}.$ 
\end{corollary}

\begin{proof}
With $\alpha=\max\{\deg(m_j) \mid j\in [q]\}$, \cref{c:shifts} gives that for all $r>0$ and all $i\le r$ there exists an integer $c_i$ such that 
\begin{equation}
\label{e:ta}
t_i^S(I^r)=\alpha r+c_i\,.
\end{equation}
In view of \cref{c:pd-I^r}, we have 
\begin{equation}
\label{e:sup}
\reg_S(I^r)=\sup_{0\le i\le q-1}\{t_i^S(I^r)-i\}\,.
\end{equation}
 If $r\ge q-1$, then for any $i$ with $0\le i\le q-1$ we also have $i\le r$, and \eqref{e:ta} gives
\begin{equation}
\label{e:t}
t_i^S(I^r)=\alpha r+c_i=(r-q+1)\alpha+\left((q-1)\alpha+c_i\right)=(r-q+1)\alpha +t_i^S(I^{q-1})\,.
\end{equation}
The desired conclusion follows from \eqref{e:sup} and \eqref{e:t}. 
\end{proof}


  \bigskip
  
\subsubsection*{Acknowledgements} 

The research leading to this paper was initiated during the week-long
workshop ``Women in Commutative Algebra'' (19w5104) which took place
at the Banff International Research Station (BIRS). The authors would
like to thank the organizers and acknowledge the
hospitality of BIRS and the additional support provided by the
National Science Foundation (NSF), DMS-1934391.

For this work Liana \c Sega and Sandra Spiroff were supported in
part by grants from the Simons Foundation (\#354594, \#584932,
respectively), and Susan Cooper and Sara Faridi were
supported by Natural Sciences and Engineering Research Council of
Canada (NSERC).

The authors are grateful to Volkmar Welker for useful background
information. The computations for this project were done using the
computer algebra software Macaulay2~\cite{M2}.  Finally,
  the authors thank both referees for carefully reading the paper and
  providing many insightful comments.

For the last author this material is based upon work supported by and
while serving at the National Science Foundation. Any opinion,
findings, and conclusions or recommendations expressed in this
material are those of the authors and do not necessarily reflect the
views of the National Science Foundation.

\subsubsection*{Data Availability} Data sharing not applicable to this
article as no datasets were generated or analysed during the current
study.

\bibliographystyle{plain}

\bibliography{bibliography}

\end{document}